\documentclass{amsart}
\usepackage{graphicx}
\usepackage{amssymb}
\usepackage{amsmath}
\usepackage{amsthm,amsfonts,bbm}
\usepackage{amscd}
\usepackage{geometry}
\usepackage[all,2cell]{xy}
\usepackage{epsfig,epstopdf}

\newtheorem{thm}{Theorem}[section]
\newtheorem{cor}[thm]{Corollary}
\newtheorem{lem}[thm]{Lemma}
\theoremstyle{definition}
\newtheorem{defi}[thm]{Definition}
\theoremstyle{remark}
\newtheorem{rmk}[thm]{\bf Remark}

\numberwithin{equation}{section}
\numberwithin{figure}{section}
\def\A{\mathcal{A}}

\def \e{\epsilon}

\def \F{\mathcal{F}}
\def \H{\mathcal{H}}

\def\la{\lambda}
\def \omg{\omega}

\def \T{\mathcal{T}}
\def \Tr{\text{Tr}}

\def\Z{\mathbb{Z}}

\begin{document}
\title[spectral characterization of graph]
{High-ordered spectral characterization of unicyclic graphs}

\author[Y.-Z. Fan]{Yi-Zheng Fan*}
\address{Center for Pure Mathematics, School of Mathematical Sciences, Anhui University, Hefei 230601, P. R. China}
\email{fanyz@ahu.edu.cn}
\thanks{*The corresponding author.
This work was supported by National Natural Science Foundation of China (Grant No. 11871073).}

%

\author[H.-X. Yang]{Hong-Xia Yang}
\address{School of Mathematical Sciences, Anhui University, Hefei 230601, P. R. China}
\email{yanghx@stu.ahu.edu.cn}

\author[J. Zheng]{Jian Zheng}
\address{School of Mathematical Sciences, Anhui University, Hefei 230601, P. R. China}
\email{zhengj@stu.ahu.edu.cn}
%
%
%
%
%


\begin{abstract}
In this paper we will apply the tensor and its traces to investigate the spectral characterization of unicyclic graphs.
Let $G$ be a graph and $G^m$ be the $m$-th power (hypergraph) of $G$.
The spectrum of $G$ is referring to its adjacency matrix, and the spectrum of $G^m$ is referring to its adjacency tensor.
The graph $G$ is called determined by high-ordered spectra (DHS for short) if, whenever $H$ is a graph such that $H^m$ is cospectral with $G^m$ for all $m$, then $H$ is isomorphic to $G$.
In this paper we first give formulas for the traces of the power of  unicyclic graphs, and then provide some
high-ordered cospectral invariants of unicyclic graphs.
We prove that a class of unicyclic graphs with cospectral mates is DHS,
and give two examples of infinitely many pairs of cospectral unicyclic graphs but with different high-ordered spectra.
\end{abstract}

\maketitle
\section{Introduction}
The graph isomorphism problem is one of few standard problems in computational complexity theory belonging to NP.
In 1956, G\"unthard and Primas \cite{GunP} raised the question of determining the graphs by the spectrum.
In 1957, Collatz and Sinogowitz \cite{ColS} presented a pair of non-isomorphic cospectral trees.
In 1973, Schwenk \cite{Sch} proved that almost every tree has a cospectral mate by constructing a non-isomorphic cospectral tree for each tree of sufficiently large order.
In 1982, Godsil and McKay \cite{GodM} invented a powerful method called GM-switching, which can produce lots of pairs of cospectral graphs.
A graph $G$ is said to be \emph{determined by the spectrum} (DS for short) if,
whenever $H$ is a graph cospectral with $G$, then $H$ must be isomorphic to $G$.
All the known DS graphs have very special structures, and the techniques involved in proving them to be DS cannot be applied to general graphs; see \cite{DamH1, DamH2}.

Since a graph cannot be determined by its spectrum in general, we need more information to recognize a graph.
Note that two graphs are isomorphic if and only if their complements are isomorphic.
Wang and Xu \cite{WangX1, WangX2} and Wang \cite{Wang} applied the spectra of a graph and also its complement to investigate whether the graph is determined by its generalized spectrum.
A similar idea appears in a recent work by Chen, Sun and Bu \cite{CSBu}, who applied the spectra of the powers of a graph to characterize whether the  graph is determined by high-ordered spectra.

Let $G=(V(G),E(G))$ be a graph with vertex set $V=V(G)$ and edge set $E=E(G)$.
For an integer $m \geq 2$, the {\it $m$-th power of $G$}, denoted by $G^{m}:=(V^{m},E^{m})$,
  is defined to be the $m$-uniform hypergraph with  vertex set $V^{m}=V\cup{\{i_{e,1},\ldots,i_{e,m-2}: e\in E}\}$ and edge set
$E^{m}={\{e\cup{{\{i_{e,1},\ldots,i_{e,m-2}}}\}: e\in E}\}$, where $i_{e,1},\ldots,i_{e,m-2}$ are new vertices inserted to  each edge $e \in E$.
Note if $m=2$, $G^m=G$ and this case is trivial.
Observe that two graphs are isomorphic if and only if their $m$-th powers are isomorphic for each integer $m \ge 2$.
Chen, Sun and Bu \cite{CSBu} introduced the following notions on the spectrum of power hypergraphs.
They called the spectrum of $G^m$ the \emph{$m$-ordered spectrum} of $G$, and two graphs \emph{$m$-ordered cospectral} if they have the same $m$-ordered spectra.
A graph $G$ is called \emph{determined by high-ordered spectra} (DHS for short) if, whenever $H$ is a graph that are $m$-ordered cospectral with $G$ for all $m \ge 2$, then $H$ must be isomorphic to $G$.

Surely, a graph that is DS must be DHS, but the converse does not hold.
For example, van Dam and Haemers \cite{DamH2} showed that not all Smith's graphs are DS.
However, all Smith's graphs are DHS, proved by Chen, Sun and Bu \cite{CSBu}.
They also showed that every tree and its cospectral mate in Schwenk's construction have different high-ordered spectra \cite{CSBu}.

The traces or spectral moments play an important role in DS problems.
Let $G$ be a graph and let $\A(G)$ be the adjacency matrix of $G$.
The $d$-th \emph{trace} (or $d$-th \emph{spectral moment}) of $G$, denoted by $\Tr_d(G)$, is defined to be the trace of $\A(G)^d$, which is the sum of the $d$-th powers of all eigenvalues of $\A(G)$, and is also equal to the number of closed walks of length $d$ in $G$.
It is known that two graphs $G$ and $H$ are cospectral if and only if $\Tr_d(G)=\Tr_d(H)$ for all $d$ (or for $d=1,2,\ldots,|V(G)|$).
We should note that here the spectrum of a uniform hypergraph is defined as the spectrum of the adjacency tensor of the hypergraph.
By the traces generalized from matrices to tensors due to Morozov and Shakirov \cite{MS2011}, we still have the above equalities for two cospectral uniform hypergraphs.
A key problem is how to interpreter the structural information from the traces of hypergraphs.


In this paper, we will extend the work of Chen, Sun and Bu \cite{CSBu} from trees to unicyclic graphs, and characterize the unicylic graph that are DHS.
The paper is organized as follows.
In Section 2 we introduce some preliminary knowledge about the spectra and traces of hypergraphs.
In Section 3 we give formulas for the traces of the power of unicyclic graphs by means of the sub-structure of the graph.
In the last section, we provide some high-ordered cospectral invariants for general graphs especially for unicyclic graphs, and prove that a class of unicyclic graphs with copectral mates is DHS. We give two examples of infinitely many pairs of cospectral unicyclic graphs but with different high-ordered spectra.
Our work implies that high-ordered spectra of graphs can recognize more structural information than the usual spectra.

\section{Preliminaries}
\subsection{Tensors and hypergraphs}
Let $\T=(t_{i_{1} i_2 \ldots i_{m}})$ be a complex tensor of order $m$ and dimension $n$.
Given a vector $x\in \mathbb{C}^{n}$, $\T x^{m-1} \in \mathbb{C}^n$, which is defined as follows:
   \[
      (\T x^{m-1})_i =\sum_{i_{2},\ldots,i_{m}\in [n]}t_{ii_{2}\ldots i_{m}}x_{i_{2}}\cdots x_{i_m}, i \in [n].
  \]
 Let $\mathcal{I}=(i_{i_1i_2\ldots i_m})$ be the {\it identity tensor} of order $m$ and dimension $n$, that is, $i_{i_{1}i_2 \ldots i_{m}}=1$ if
   $i_{1}=i_2=\cdots=i_{m} \in [n]$ and $i_{i_{1}i_2 \ldots i_{m}}=0$ otherwise.
In 2005 Lim \cite{Lim} and Qi \cite{Qi} introduced the eigenvalues of tensors independently as follows.

\begin{defi}[\cite{Lim,Qi}]\label{eigen} Let $\T$ be an $m$-th order $n$-dimensional tensor.
For some $\lambda \in \mathbb{C}$, if the polynomial system $(\lambda \mathcal{I}-\T)x^{m-1}=0$, or equivalently $\T x^{m-1}=\lambda x^{[m-1]}$, has a solution $x\in \mathbb{C}^{n}\backslash \{0\}$,
then $\lambda $ is called an \emph{eigenvalue} of $\T$ and $x$ is an \emph{eigenvector} of $\T$ associated with $\lambda$,
where $x^{[m-1]}:=(x_1^{m-1}, x_2^{m-1},\ldots,x_n^{m-1})$.
\end{defi}


A  \emph{hypergraph} $\H=(V,E)$ consists of a vertex set $V=\{v_1,v_2,{\cdots},v_n\}$ denoted by $V(\H)$ and an edge set $E=\{e_1,e_2,{\cdots},e_k\}$ denoted by $E(\H)$,
 where $e_i \subseteq V$ for $i \in [k]$.
 If $|e_i|=m$ for each $i \in [k]$ and $m \geq2$, then $\H$ is called an \emph{$m$-uniform} hypergraph.
 The \emph{degree} $d_v(\H)$ of a vertex $v$ in $\H$ is the number of edges of $\H$ containing the vertex $v$.
A vertex $v$ of $\H$ is called a \emph{cored vertex} if it has degree one.
A {\it walk} $W$ in $\H$ is a sequence of alternate vertices and edges: $v_{0}e_{1}v_{1}e_{2}\cdots e_{l}v_{l}$,
    where $v_{i} \ne v_{i+1}$ and $\{v_{i},v_{i+1}\}\subseteq e_{i}$ for $i=0,1,\ldots,l-1$.
If $v_0=v_l$, then $W$ is called a {\it circuit}, and is called a {\it cycle} if no vertices or edges are repeated except $v_0=v_l$.
The  hypergraph $\H$ is said to be {\it connected} if every two vertices are connected by a walk.
  The hypergraph $\H$ is called \emph{simple} if there exists no $i \ne j$ such that $e_i \subseteq e_j$, and is called \emph{nontrivial} if it contains more than one vertex.
Throughout of this paper, \emph{all hypergraphs are considered nontrivial, connected, simple and $m$-uniform unless stated somewhere.}


In 2012 Cooper and Dutle \cite{CD2012} introduced the adjacency tensor of a uniform hypergraph, and applied the eigenvalues of the tensor to characterize the structural property of the hypergraph.

\begin{defi}[\cite{CD2012}]\label{def-adj}
Let $\H$ be an $m$-uniform hypergraph on $n$ vertices $v_1,v_2,\ldots,v_n$.
The \emph{adjacency tensor} of $\H$ is defined as $\mathcal{A}(\H)=(a_{i_{1}i_{2}\ldots i_{m}})$, an $m$-th order $n$-dimensional tensor, where
$$a_{i_1 i_2 \ldots i_m}=\left\{
\begin{array}{cl}
\frac{1}{(m-1)!}, & \mbox{~if~} \{v_{i_1},\ldots,v_{i_m}\} \in E(H);\\
0, & \mbox{~else}.
\end{array}\right.
$$
\end{defi}
The \emph{spectrum} and \emph{eigenvalues} of $\H$ are referring to those of $\A(\H)$.
If $m=2$, then $\A(\H)$ is exactly the adjacency matrix of the graph $\H$.
Since the Perron-Frobenius theorem of nonnegative matrices was generalized to  nonnegative tensors \cite{CPZ1,FGH,YY1,YY2,YY3},
 the spectral hypergraph theory develops rapidly on many topics, such as the spectral radius \cite{BL,FanTPL, GCH2022,KLM2014,LSQ, LKS2018,LM}, the eigenvariety \cite{FBH,FTL,FHBproc}, the spectral symmetry \cite{FHB,FLW, SQH2015,Zhou}, the eigenvalues of hypertrees \cite{ZKSB}.

Let $G^m$ be the $m$-th power of a graph $G$.
Zhou et al. \cite{Zhou} showed that all eigenvalues of $G^m$ can be obtained from the eigenvalues of $G$.
Cardoso et al. \cite{CardHT} gave a more detailed statement as follows.

\begin{lem} [\cite{CardHT}]\label{Card}
Let $G^m$ be the $m$-th power of a graph $G$.

\begin{itemize}
\item[(1)] If $m=3$, $\la$ is an eigenvalue of $G^3$ if and only if there is an induced subgraph of $G$ with eigenvalue $\beta$ such that $\beta^2=\la^3$.

\item[(2)] If $m \ge 4$, $\la$ is an eigenvalue of $G^m$ if and only if there is a subgraph of $G$ with eigenvalue $\beta$ such that $\beta^2=\la^m$.
\end{itemize}

\end{lem}

From Lemma \ref{Card}, we find that $G^m$ contains more spectral information than $G$ as its eigenvalues are closely related to all (induced) subgraphs of $G$ (including $G$ itself).
So, it is natural that a graph has more probability to be DHS than to be DS.

\subsection{Traces}
We will introduce some knowledge about the traces of hypergraphs.
Let $\H$ be an $m$-uniform hypergraph on $n$ vertices.
The \emph{$d$-th trace of $\H$}, denoted by $\Tr_d(\H)$, is referring to the $d$-th trace of $\A(\H)$.
Morozov and Shakirov \cite{MS2011} introduced the traces of polynomial maps $f$ given by homogeneous polynomials of arbitrary degrees.
As a tensor $\T=(t_{i_1i_2 \ldots i_m})$ of order $m$ and dimension $n$ naturally induces polynomial maps,
the $d$-th trace $\Tr_d(\T)$ of $\T$  is expressed as follow:
\begin{equation}\label{MSeq}\Tr_d(\T)=(m-1)^{n-1} \sum_{d_1+\cdots+d_n=d, \atop d_i \in \mathbb{N}, i \in [n]}
\prod_{i=1}^n \frac{1}{(d_i(m-1))!} \left(\sum_{y_i \in [n]^{m-1}}t_{iy_i}\frac{\partial }{\partial a_{iy_i}}\right)^{d_i}
\Tr(A^{d(m-1)}),
\end{equation}
where $t_{iy_i}=t_{ii_2 \ldots i_m}$ and $\frac{\partial }{\partial a_{iy_i}}=\frac{\partial }{\partial a_{ii_2}} \cdots \frac{\partial }{\partial a_{ii_m}}$ if
$y_i=(i_2, \ldots, i_m)$.

Cooper and Dulte \cite{CD2012} gave an expression for the co-degree coefficients of the characteristic polynomial of $\A(\H)$ of $\H$ in terms of traces of $\H$.
Shao, Qi and Hu \cite{SQH2015} gave a graph interpretation for the $d$-th trace of a general tensor $\T$ of order $m$ and dimension $n$, and proved that
$$\Tr_d(\T)=\sum_{i=1}^N \la_i^d,$$
where $\la_1,\ldots,\la_N$ are all eigenvalues of $\T$, and $N=n(m-1)^{n-1}$.
So two tensors $\T_1$ and $\T_2$ both of order $m$ and dimension $n$ are cospectral if and only if
$$ \Tr_d(\T_1)=\Tr_d(\T_2)$$
for all $d$ or $d=1,2,\ldots,n(m-1)^{n-1}$.
Clark and Cooper \cite{CC2021} expressed the trace as a weighted sum over a family of Veblen hypergraphs.
Chen, Bu and Zhou \cite{CBZ2022} gave a formula for the spectral moments (equivalently, the traces) of a hypertree in terms of the number of sub-hypertrees.

Given an ordering of the vertices of $\H$, let
$$\F_d(\H):=\{(e_1(v_1),\ldots,e_d(v_d)): e_i \in E(\H), v_1 \le \cdots \le v_d\},$$
be the set of $d$-tuples of ordered rooted edges, where $e_i(v_i)$ is an edge $e_i$ with root $v_i \in e_i$ for $i \in [d]$.
Define a rooted directed star $S_{e_i}(v_i)=(e_i, \{(v_i,u): u \in e_i\backslash \{v_i\}\})$ for each $i \in [d]$, and multi-directed graph
$R(F)=\bigcup_{i=1}^d S_{e_i}(v_i)$ associated with $F \in \F_d(\H)$.
Let
$$ \F^\e_d(\H):=\{F \in \F_d(\H): R(F) ~is ~Eulerian\}.$$
For an $F \in \F_d^\e(\H)$, denote $V(F):=V(R(F))$, $r_v(F)$ the number of edges in $F$ with $v$ as the root, and $d_v^+(F)=(m-1)r_v(F)$ (namely, the outdegree of $v$ in $R(F)$).
Denote by $\tau(F):=\tau_u(R(F))$ the number of arborescences of $R(F)$ with root $u$ (namely, a directed
$u$-rooted spanning tree such that all vertices except $u$ has a directed path from itself to $u$), which is equal to the principal minor the Laplacian matrix $L(R(F))$ of $R(F)$ by deleting the row and column indexed by $u$ (\cite{TS1941, EB1987}).
As $R(F)$ is Eulerian,
$ \tau_u(R(F))$ is independent of the choice of the root $u$ so that the root $u$ is omitted.
Fan et al. \cite{Fan2022} give an expression of the $d$-th trace of $\H$ as follows.

\begin{lem}[\cite{Fan2022}] \label{TrF} For an $m$-uniform hypergraph $\H$ on $n$ vertices,
\begin{equation}\label{Shaoeq_Hyp3}\Tr_d(\H)=d(m-1)^{n}\sum_{F \in \F_d^\e(\H)} \frac{\tau(F)}{\prod_{v \in V(F)} d_v^+(F)}.
\end{equation}
\end{lem}



For each $F \in \F_d^\e(\H)$, we get a multi-hypergraph induced by the edges in $F$ by omitting the roots, denoted by $\mathcal{V}_F$, is an $m$-uniform and $m$-valent multi-hypergraph called \emph{Veblen hypergraph}.
On the other side,
given a Veblen hypergraph $H$, a \emph{rooting} of $H$ is an ordering $F=(e_1(v_1),\ldots,e_t(v_t))$ of all edges of $H$, where $v_i$ is the root of $e_i$ for $i \in [t]$, and $v_1 \le \cdots \le v_t$ under the given order of the vertices of $H$.
If $R(F)$ is Eulerian, then $F$ is called an \emph{Euler rooting} of $H$;
in this case, $H$ is called \emph{Euler rooted} with each edge rooted as in $F$ by omitting the order.
Denote by $\mathcal{R}(H)$ the set of Euler rooting of $H$.

Denote by $\mathcal{V}_d(\H)$ the set of Veblen hypergraphs with $d$ edges associated with $\H$ as follows:
$$ \mathcal{V}_d(\H)=\cup_{G \in \mathcal{C}(\H)}\{\mathcal{V}_F: F \in  \F_d^\e(\H), \underline{\mathcal{V}_F}=G\},$$
where $\mathcal{C}(\H)$ denotes the set of representatives of the isomorphic classes of connected sub-hypergraphs of $\H$, and
$\underline{H}$ the underlying hypergraph of a multi-hypergraph $H$ obtained by  removing duplicate edges of $H$.
For each $H \in \mathcal{V}_d(\H)$, denote
$$
C_H=\sum_{F \in \mathcal{R}(H)}\frac{\tau(F)}{\prod_{v \in V(F)} d_v^+(F)},$$
and $N_{\H}(\underline{H})$ the number of sub-hypergraphs of $\H$ that is isomorphic to $\underline{H}$.
By Lemma \ref{TrF}, we have
\[
\Tr_d(\H)=d(m-1)^{n} \sum_{H \in \mathcal{V}_d(\H)} \left(\sum_{F \in \mathcal{R}(H)}\frac{\tau(F)}{\prod_{v \in V(F)} d_v^+(F)}\right) N_{\H}(\underline{H}).
\]
So we get another expression of $\Tr_d(\H)$.

\begin{cor}\label{TrV} For an $m$-uniform hypergraph $\H$ on $n$ vertices,
\begin{equation}\label{Shaoeq_Hyp3}\Tr_d(\H)=d(m-1)^{n}\sum_{H \in \mathcal{V}_d(\H)} C_H N_{\H}(\underline{H}).
\end{equation}
\end{cor}

If we use $\overline{\mathcal{V}}_d(\H)$ denote the set of representatives of the isomorphism classes of the hypergraphs in  $\mathcal{V}_d(\H)$, then we have the following expression of $\Tr_d(\H)$, which was proved by Clark and Cooper \cite{CC2021}.

\begin{equation}\label{Clark2}\Tr_d(\H)=d(m-1)^{n} \sum_{H \in \overline{\mathcal{V}}_d(\H)}C_H N_{\H}(\underline{H}) \frac{|\text{Aut}(\underline{H})|}{|\text{Aut}(H)|}:=d(m-1)^{n} \sum_{H \in \overline{\mathcal{V}}_d(\H)}C_H (\#H \subseteq \H),
\end{equation}
where
$$(\#H \subseteq \H)=N_{\H}(\underline{H}) \frac{|\text{Aut}(\underline{H})|}{|\text{Aut}(H)|},$$
and $\text{Aut}(G)$ denotes the automorphism group of a hypergraph $G$.

\subsection{Traces of hypertrees}\label{TrTree}
A \emph{hypertree} is a connected and acyclic hypergraph.
Let $\T$ be an $m$-uniform hypertree.
Then we have $|V(\T)|=(m-1)|E(\T)|+1$.
We need the following lemma to characterize the Veblen hypergraphs $H \in \mathcal{V}_d(\T)$.

\begin{lem}[\cite{Fan2022}]\label{tree_root}
Let $H$ be an $m$-uniform Veblen multi-hypergraph whose underlying hypergraph $\underline{H}$ is a hypertree.
Then $H$ is uniquely Euler rooted such that all vertices of each edge occur as roots of the edge in a same number of times, and hence every edge of $H$ repeats in a multiple of $m$ times.
\end{lem}

By Lemma \ref{tree_root}, for $d \in \mathbb{Z}^+$, $\mathcal{V}_d(\T) \ne \emptyset$ if and only if $m \mid d$.
So, in this subsection we always assume that \emph{$d$ is a positive multiple of $m$} when discussing $\Tr_d(\T)$; otherwise, $\Tr_d(\T)=0$ by Corollary \ref{TrV}.
For each $H \in \mathcal{V}_d(\T)$, let $\underline{H}=\hat{\T}$, a sub-hypertree of $\T$.
Then $H$ can be expressed as a weighted hypertree $\hat{\T}(\omega)$, where
$$ \omega: E(\hat{\T}) \to \Z^+,$$
such that the multiplicity of an edge $e \in E(H)$ is $m\omega(e)$, and $\omg(\hat{\T}):=\sum_{e \in E(\hat{\T})} \omega(e)=d/m$, implying that $\hat{\T}$ contains at most $d/m$ edges.
So we have
$$\mathcal{V}_d(\T)=\cup_{\hat{\T}\in \mathcal{C}(\T)}\{\hat{\T}(\omega): \omg(\hat{\T})=d/m\}.$$

For each $F \in \mathcal{R}(\hat{\T}(\omega))$, by a direct computation, we have
$$ \tau(F)=\prod_{e \in E(\hat{\T})} \omg(e)^{m-1} m^{m-2}=m^{(m-2)|E(\hat{\T})|} \left(\prod_{e \in E(\hat{\T})}\omg(e)\right)^{m-1},$$
$$\prod_{v \in V(F)}d_v^+(F)=\prod_{v \in V(\hat{\T})} (m-1)d_v(\hat{\T}(\omg))=(m-1)^{|V(\hat{\T})|}\prod_{v \in V(\hat{\T})} d_v(\hat{\T}(\omg)),$$
where $d_v(\hat{\T}(\omega))=\sum_{e: v \in e} \omega(e)$, the weighted degree of the vertex $v$ in $\hat{\T}(\omega)$.

By Lemma \ref{tree_root}, for each $F \in \mathcal{R}(\hat{\T}(\omega))$, every vertex $v \in V(\hat{\T})$ occurs as root in $d_v(\T(\omega))$ times of $d_v(\hat{\T})$ distinct edges $e$ with multiplicity $\omg(e)$ respectively.
So $\mathcal{R}(\hat{\T}(\omega))$ has
$\displaystyle \prod_{v \in V(\hat{\T})} \frac{d_v(\hat{\T}(\omega))!}{r_v(\hat{\T}(\omega))}$
Euler rootings due to the ordering of the same roots in different edges, where $r_v(\hat{\T}(\omega)=\prod_{e: v \in e} \omega(e)!$.
So, for a given $H=\hat{\T}(\omega) \in \mathcal{V}_d(\T)$,
\begin{align}\label{SpeT}
\begin{split}
 C_H & =  \prod_{v \in V(\hat{\T})} \frac{d_v(\hat{\T}(\omega))!}{r_v(\hat{\T}(\omega))} \cdot \frac{\tau(F)}{\prod_{v \in V(F)} d_v^+(F)}\\
 & = (m-1)^{-|V(\hat{\T})|}m^{(m-2)|E(\hat{\T})|} \left(\prod_{e \in E(\hat{\T})}\omg(e)\right)^{m-1} \prod_{v \in V(\hat{\T})} \frac{(d_v(\hat{\T}(\omega))-1)!}{r_v(\hat{\T}(\omega))}.
 \end{split}
\end{align}

Denote
$$ c_{d,m}(\hat{\T})= \sum_{\omg: \omg(\hat{\T})=d/m} \left(\prod_{e \in E(\hat{\T})}\omg(e)\right)^{m-1} \prod_{v \in V(\hat{\T})} \frac{(d_v(\hat{\T}(\omega))-1)!}{r_v(\hat{\T}(\omega))}.$$
Denote by $\mathbf{T}^m_{\le t}$ (respectively, $\mathbf{T}^m_{t}$) the set of $m$-uniform hypertrees with at most $t$ edges (respectively, exactly $t$ edges) up to isomorphism.
Then by Corollary \ref{TrV}, for $m \mid d$,
\begin{equation}\label{Bu1}
\begin{split}
\Tr_d(\T)& =d(m-1)^{|V(\T)|} \sum_{H \in \mathcal{V}_d(\T)}C_H N_{\T}(\underline{H})\\
& =\sum_{\hat{\T}(\omg): \hat{\T} \in \mathbf{T}^m_{\le d/m},  \omg(\hat{\T})=d/m} d(m-1)^{{|V(\T)|-|V(\hat{\T})|}}m^{(m-2)|E(\hat{\T})|} \\
& \quad \cdot \left(\prod_{e \in E(\hat{\T})}\omg(e)\right)^{m-1} \prod_{v \in V(\hat{\T})} \frac{(d_v(\hat{\T}(\omega))-1)!}{r_v(\hat{\T}(\omega))} \cdot N_{\T}(\hat{\T})\\
&= \sum_{\hat{\T} \in \mathbf{T}^m_{\le d/m}} d(m-1)^{{|V(\T)|-|V(\hat{\T})|}}m^{(m-2)|E(\hat{\T})|}\\
&\quad \cdot
\sum_{\omg: \omg(\hat{\T})=d/m} \left(\prod_{e \in E(\hat{\T})}\omg(e)\right)^{m-1} \prod_{v \in V(\hat{\T})} \frac{(d_v(\hat{\T}(\omega))-1)!}{r_v(\hat{\T}(\omega))}\cdot N_{\T}(\hat{\T})\\
& = \sum_{\hat{\T} \in \mathbf{T}^m_{\le d/m}} d(m-1)^{(m-1)(|E(\T)|-|E(\hat{\T})|)}m^{(m-2)|E(\hat{\T})|}
c_{d,m}(\hat{\T}) N_{\T}(\hat{\T})\\
&=\sum_{k=1}^{d/m} d(m-1)^{(m-1)(|E(\T)|-k|)}m^{k(m-2)} \sum_{\hat{\T} \in \mathbf{T}^m_k} c_{d,m}(\hat{\T}) N_{\T}(\hat{\T}).
\end{split}
\end{equation}

Now suppose that $\T=T^m$, the $m$-th power of a tree $T$.
Then $H=\hat{\T}(\omg)=\hat{T}^m(\omg)$ for some sub-tree $\hat{T}$ of $T$.
For each edge $e=\{u,v\} \in E(T)$, it corresponds to an edge $e^m:=\{u,v,e_{i1},\ldots,e_{i,m-2}\} \in E(T^m)$.
The weighted hypertree $\hat{T}^m(\omg)$ induces a weighted tree $\hat{T}(\omg)$ with weight $\omg: E(\hat{T}) \to \Z^+$ such that $\omg(e)=\omg(e^m)$ and $\omg(\hat{T}):=\sum_{e \in E(\hat{T})} \omg(e)=d/m$.
We have
$$ \prod_{v \in V(\hat{T}^m)} \frac{(d_v(\hat{T}^m(\omega))-1)!}{r_v(\hat{T}^m(\omega))}=
\prod_{v \in V(\hat{T})} \frac{(d_v(\hat{T}(\omega))-1)!}{r_v(\hat{T}(\omega))}
\prod_{e \in E(\hat{T})}\left(\frac{(\omg(e)-1)!}{\omg(e)!}\right)^{m-2}.
$$
Hence
\begin{equation}\label{cdom} c_{d,m}(\hat{T}^m)= \sum_{\omg:\omg(\hat{T})=d/m} \prod_{e \in E(\hat{T})}\omg(e) \prod_{v \in V(\hat{T})} \frac{(d_v(\hat{T}(\omega))-1)!}{r_v(\hat{T}(\omega))}=:\tilde{c}_{d/m}(\hat{T}).
\end{equation}
So, for $m \mid d$,
\begin{equation}\label{Bu2}
\Tr_d(T^m) =\sum_{k=1}^{d/m} d(m-1)^{(m-1)(|E(T)|-k)}m^{k(m-2)} \sum_{\hat{T} \in \mathbf{T}_k} \tilde{c}_{d/m}(\hat{T}) N_{T}(\hat{T}).
\end{equation}

We note that Eq. (\ref{Bu1}) and Eq. (\ref{Bu2}) with a slightly modification were proved by Chen, Bu and Zhou \cite{CBZ2022}. Here we show them in a different way also for the convenience in the following discussion.

\section{Traces of power of unicyclic graphs}
In this section, we will give a decomposition formula for the traces of the power of unicyclic graphs.
Denote by $C_n$ a cycle on $n$ vertices (as a graph).
Let $U^m$ be the $m$-th power of a unicyclic graph $U$ which contains a cycle $C_n$, where $m \ge 3$.
The following lemma is important for our discussion.

\begin{lem}[\cite{Fan2022}]\label{core}
Let $H$ be an $m$-uniform Veblen multi-hypergraph, and let $e$ be an edge of $H$  which contains a cored vertex.
If $H$ has an Euler rooting, then $e$ repeats $k \cdot m$ times for some positive integer $k$, and all cored vertices in $e$ occur as a root of $e$ in $k$ times.
\end{lem}

Let $H \in \mathcal{V}_d(U^m)$.
As $m \ge 3$, each edges of $H$ contains $m-2$ cored vertices.
 By Lemma \ref{core}, each edge of $H$ occurs in a multiple of $m$ times.
  So, for $d \in \mathbb{Z}^+$,  $\mathcal{V}_d(U^m) \ne \emptyset$ if and only if $m \mid d$.
We will assume $m \mid d$ in this section.

\subsection{Traces of power of cycles}\label{powC}
We first consider a special case, namely, $U=C_n$, where $C_n$ has vertices $v_1,\ldots, v_n$ and edges $e_i=\{v_i,v_{i+1}\}$ for $i \in [n]$, $v_{n+1}=v_1$.
We label the edges of $C_n^m$ as $e^m_i=\{v_i,v_{i+1},e_{i1}, \ldots, e_{i,m-2}\}$ for $i \in [n]$.
Let
$$\mathcal{V}_d(C_n^m;[C_n^m])=\{ H\in \mathcal{V}_d(C_n^m): \underline{H}=C_n^m\}.$$
Let $H \in \mathcal{V}_d(C_n^m;[C_n^m])$.
 By Lemma \ref{core},
$H$ is a weighted cycle $C_n^m(\omg)$ with $\omg: E(C_n^m) \to \Z^+$ such that
$e^m_i$ repeats $m \omg(e^m_i)$ times for $i \in [n]$, where $\omg(C_n^m):=\sum_{i=1}^n\omg(e^m_i)=d/m$.
So
$$\mathcal{V}_d(C_n^m;[C_n^m])=\{ C_n^m(\omg): \omg(C_n^m)=d/m\}.$$

Let $H = C_n^m(\omg)$ and $F \in \mathcal{R}(C_n^m(\omg))$.
It is known that each cored vertex of $e^m_i$ occurs as a root of $e^m_i$ in $\omg(e^m_i)$ times for $i \in [n]$ by Lemma \ref{core}.
Let $t_{ii}$ (respectively, $t_{i,i+1}$) be the times of $v_i$ (respectively, $v_{i+1}$) as a root of $e^m_i$ for $i \in [n]$, where the subscripts are taken modulo $n$.
As $e^m_i$ repeats $m \omg(e^m_i)$ times and each cored vertex of $e^m_i$ occurs as a root of $e^m_i$ in $\omg(e^m_i)$ times,
we have
\begin{equation}\label{Rcycle1} t_{ii}+t_{i,i+1}=2\omg(e^m_i).
\end{equation}
As $R(F)$ is Eulerian,
$$ (t_{i-1,i}+t_{ii})(m-1)=m \omg(e^m_{i-1})-t_{i-1,i}+m \omg(e^m_{i})-t_{ii},$$
which implies that
\begin{equation}\label{Rcycle2} t_{i-1,i}+t_{i,i}=\omg(e^m_{i-1})+\omg(e^m_i).
\end{equation}
Suppose that $i_0 \in [n]$ such that $\omg(e^m_{i_0})=\min_{i \in [n]} \omg(e^m_i)=:\omg_{\min}$.
By Eqs. (\ref{Rcycle1}) and (\ref{Rcycle2}), we have
$$ t_{ii}=\omg(e^m_i)-\omg(e^m_{i_0})+t_{i_0i_0},~
t_{i,i+1}=\omg(e^m_i)+\omg(e^m_{i_0})-t_{i_0i_0}, t_{i_0i_0} \in [0, 2\omg(e^m_{i_0})],$$
or equivalently
$$ t_{ii}=\omg(e^m_i)-\omg_{\min}+x,~
t_{i,i+1}=\omg(e^m_i)+\omg_{\min}-x, x \in [0, 2\omg_{\min}].$$
So
$$\mathcal{R}(C_n^m(\omg))=\cup_{x=0}^{2 \omg_{\min}} \mathcal{R}(C_n^m(\omg);x),$$
where $\mathcal{R}(C_n^m(\omg);x)$ consists of those rootings $F \in \mathcal{R}(C_n^m(\omg))$ such that
each cored vertex of $e^m_i$  acts as a root of $e^m_i$  in $\omg(e^m_i)$ times, and for $i \in [n]$,
$v_i$ (respectively, $v_{i+1}$) acts as a root of $e^m_i$ in $t_{ii}=\omg(e^m_i)-\omg_{\min}+x$ (respectively, $t_{i,i+1}=\omg(e^m_i)+\omg_{\min}-x$) times.

For each $F \in \mathcal{R}(C_n^m(\omg);x)$, we now calculate $\tau(F)$.
Note that $ \tau(F)$ is the principal minor of $L(R(F))$ by deleting the row and column both indexed by a specified vertex $v$.
Here we take $v=v_1$ and write the $L(R(F))$ as follows, where $t_i:=(t_{i-1,i}+t_{ii})(m-1)$,
$\omg_i:=\omg(e^m_i)$ for $i \in [n]$, $\mathbf{1}$ is an all-one's column vector of size $m-2$, $K=mI-J$ of size $m-2$, and $I$ is an identity matrix and $J$ is an all-one's matrix.
{\small
\begin{align*}
& \left[
\begin{array}{c|cccc|ccccc}
t_1     & -t_{11} &    0     & \cdots  & -t_{n1} & -t_{11}\mathbf{1}^\top & 0 & 0& \cdots & -t_{n1}\mathbf{1}^\top\\
\hline
-t_{12} & t_2     & -t_{22} & \cdots   & 0 & -t_{12}\mathbf{1}^\top &  -t_{22}\mathbf{1}^\top & 0& \cdots & 0\\
     \vdots   &  \ddots       &  \ddots       & \ddots & \vdots  & \vdots &  \ddots& \ddots  &\vdots & \vdots \\
0 & \cdots & -t_{n-2,n-1} & t_{n-1} & -t_{n-1,n-1} & 0&  \cdots&  -t_{n-2,n-1}\mathbf{1}^\top   & -t_{n-1,n-1}\mathbf{1}^\top &0\\
-t_{nn} & \cdots & 0& -t_{n-1,n} & t_n & 0 &\cdots & 0& -t_{n-1,n}\mathbf{1}^\top   & -t_{n,n}\mathbf{1}^\top\\
\hline
 -\omg_1 \mathbf{1} &  -\omg_1 \mathbf{1} & 0 & \cdots & 0 & \omg_1K &  O & O & \cdots & O \\
 0 &   -\omg_2 \mathbf{1} &  -\omg_2 \mathbf{1}  &  \cdots & 0 & O & \omg_2K & O & \cdots &O  \\
 \vdots  & \vdots  & \ddots   & \ddots & \vdots & \vdots & \vdots & \ddots & \vdots&O \\
  0 & \cdots  &\cdots &-\omg_{n-1} \mathbf{1} & -\omg_{n-1} \mathbf{1}& O & \cdots& O & \omg_{n-1}K & O \\
 -\omg_n \mathbf{1} & \cdots  & \cdots & 0 & -\omg_n \mathbf{1}& O & \cdots& O &O &\omg_nK
\end{array}
\right]\\
& := \left[\begin{array}{ccc}
t_1 & * & * \\
* & A & B \\
* & C & D
\end{array}
\right].
\end{align*}
}
So
$$
\tau(F)= \det \left[\begin{array}{cc}
A & B \\
C& D
\end{array} \right]\\
= \det D \det (A-BD^{-1}C).$$
We easily get
$$ \det D=2^n m^{n(m-3)} \left(\prod_{i=1}^n \omg_i\right)^{m-2}.$$
Note that $\mathbf{1}^\top K=2 \mathbf{1}^\top$ so that $\mathbf{1}^\top K^{-1}=\frac{1}{2}\mathbf{1}^\top$.
We have
$$
BD^{-1}C=\frac{m-2}{2}
\left[
\begin{array}{ccccc}
\frac{t_2}{m-1} & t_{22} & 0  & \cdots &  0\\
t_{23} & \frac{t_3}{m-1} & t_{33} & \cdots & 0 \\
 \vdots & \ddots & \ddots & \ddots &  \vdots\\
0 & \cdots & t_{n-2,n-1} & \frac{t_{n-1}}{m-1} & t_{n-1,n-1}\\
0 & \cdots & 0 & t_{n-1,n} & \frac{t_n}{m-1}
 \end{array}
 \right],
$$
and hence
$$
A-BD^{-1}C = \frac{m}{2}
\left[
\begin{array}{ccccc}
t_{12}+t_{22} & -t_{22} & 0  & \cdots &  0\\
-t_{23} & t_{23}+t_{33} & -t_{33} & \cdots & 0 \\
 \vdots & \ddots & \ddots & \ddots &  \vdots\\
0 & \cdots & -t_{n-2,n-1} &  t_{n-2,n-1}+t_{n-1,n-1} & -t_{n-1,n-1}\\
0 & \cdots & 0 & -t_{n-1,n} & t_{n-1,n}+t_{nn}
 \end{array}
 \right].
$$
We get
$$\det (A-BD^{-1}C)=\left(\frac{m}{2}\right)^{n-1} \left(\prod_{i=2}^n t_{ii} + \sum_{l=1}^{n-2} \prod_{i=1}^l t_{i,i+1} \prod_{i=l+2}^n t_{ii} + \prod_{i=1}^{n-1}t_{i,i+1}\right)=\left(\frac{m}{2}\right)^{n-1} \sum_{l=0}^{n-1} \prod_{i=1}^l t_{i,i+1} \prod_{i=l+2}^n t_{ii},$$
and therefore
$$ \tau(F)=2 m^{n(m-2)-1} \left(\prod_{i=1}^n \omg_i\right)^{m-2} \sum_{l=0}^{n-1} \prod_{i=1}^l t_{i,i+1} \prod_{i=l+2}^n t_{ii}.$$
By the diagonal entries of $L(R(F))$, we also get
$$\prod_{v \in V(F)} d_v^+(F)=\prod_{i=1}^n t_i \prod_{i=1}^n (\omg_i(m-1))^{m-2}=(m-1)^{n(m-1)} \left(\prod_{i=1}^n \omg_i\right)^{m-2} \prod_{i=1}^n (t_{i-1,i}+t_{ii}).$$

Note $\mathcal{R}(C_n^m(\omg);x)$ contains exactly
$ \prod_{i=1}^n {t_{i-1,i}+t_{ii} \choose t_{ii}}$
rootings $F$, all corresponding the same values of $\tau(F)$ and $\prod_{v \in V(F)} d_v^+(F)$.
So we get $C_H$ for $H=C_n^m(\omg)$ as follows:
\begin{equation}\label{CC}
\begin{split}
C_{C_n^m(\omg)}&=\sum_{F \in \mathcal{R}(C_n^m(\omg))} \frac{\tau(F)}{\prod_{v \in V(F)} d_v^+(F)}\\
&= \sum_{x=0}^{2\omg_{\min}} \sum_{F \in \mathcal{R}(C_n^m(\omg);x)} \frac{\tau(F)}{\prod_{v \in V(F)} d_v^+(F)}\\
 &= \sum_{x=0}^{2\omg_{\min}} \prod_{i=1}^n {t_{i-1,i}+t_{ii} \choose t_{ii}}\cdot \frac{2 m^{n(m-2)-1} \sum_{l=0}^{n-1} \prod_{i=1}^l t_{i,i+1} \prod_{i=l+2}^n t_{ii}}{(m-1)^{n(m-1)}  \prod_{i=1}^n (t_{i-1,i}+t_{ii})}\\
 &= \frac{2 m^{n(m-2)-1}}{(m-1)^{n(m-1)}  \prod_{i=1}^n (\omg_{i-1}+\omg_{i})} \sum_{x=0}^{2\omg_{\min}}
\prod_{i=1}^n {\omg_{i-1}+\omg_{i} \choose \omg_i-\omg_{\min}+x}
\sum_{l=0}^{n-1} \prod_{i=1}^l (\omg_i+\omg_{\min}-x) \prod_{i=l+2}^n (\omg_i-\omg_{\min}+x).
\end{split}
\end{equation}

Denote
$$ f_{C_n}(\omg)=\frac{1}{\prod_{i=1}^n (\omg_{i-1}+\omg_{i})}\sum_{x=0}^{2\omg_{\min}}
\prod_{i=1}^n {\omg_{i-1}+\omg_{i} \choose \omg_i-\omg_{\min}+x}
\sum_{l=0}^{n-1} \prod_{i=1}^l (\omg_i+\omg_{\min}-x) \prod_{i=l+2}^n (\omg_i-\omg_{\min}+x)$$
and
$$\Tr_d(C_n^m; [C_n^m])=d(m-1)^{|V(C_n^m)|} \sum_{H \in \mathcal{V}_d(C_n^m; [C_n^m])} C_H N_{C_n^m}(\underline{H}).$$

\begin{lem}\label{TrCycle} For $m \mid d$,
\begin{equation}\label{Trcyc1} \Tr_d(C_n^m; [C_n^m])=\sum_{\omg: \omg(C_n^m)=d/m} 2d m^{n(m-2)-1} f_{C_n}(\omg).
\end{equation}
If $d/m=n$, then
\begin{equation}\label{Trcyc2}\Tr_{nm}(C_n^m; [C_n^m])=2n(n+1)m^{n(m-2)}.
\end{equation}
\end{lem}

\begin{proof}
By the previous discussion,
each $H \in \mathcal{V}_d(C_n^m; [C_n^m])$ is a weighted cycle $C_n^m(\omg)$, and $\underline{H}=C_n^m$, implying that $N_{C_n^m}(\underline{H})=1$.
By definition and Eq. (\ref{CC}),
\begin{align*}
\Tr_d(C_n^m; [C_n^m])&=d(m-1)^{n(m-1)}\sum_{H \in \mathcal{V}_d(C_n^m; [C_n^m])} C_H N_{C_n^m}(\underline{H})\\
& =\sum_{\omg: \omg(C_n^m)=d/m}
d(m-1)^{n(m-1)} C_{C_n^m(\omg)}\\
&  =\sum_{\omg: \omg(C_n^m)=d/m}
d(m-1)^{n(m-1)}\frac{2 m^{n(m-2)-1}}{(m-1)^{n(m-1)}} f_{C_n}(\omg) \\
&=\sum_{\omg: \omg(C_n^m)=d/m} 2d m^{n(m-2)-1} f_{C_n}(\omg).
\end{align*}

If $d/m=n$, then $\omg_i=1$ for $i \in [n]$,
$f_{C_n}(\omg)=n+1$, and therefore
$$\Tr_{nm}(C_n^m; [C_n^m])=2n(n+1)m^{n(m-2)}.$$
\end{proof}

\subsection{Trace of power of unicyclic graphs}
Let $U^m$ be the $m$-th power of a unicyclic graph $U$ which contains a cycle $C_n$. The labeling of the vertices and edges of $C_n^m$ is same as in Section \ref{powC}.
We have a decomposition:
  \begin{equation}\label{setdec}
  \mathcal{V}_d(U^m)=\mathcal{V}_d(U^m;[\hat{C}_n^m])\cup  \mathcal{V}_d(U^m;[C_n^m]),
  \end{equation}
where $\mathcal{V}_d(U^m;[\hat{C}_n^m])$ (respectively, $\mathcal{V}_d(U^m;[C_n^m])$ ) denotes the subset of  $\mathcal{V}_d(U^m)$ which consists of Veblen hypergraphs that contain not all edges of $C_n^m$ (respectively, contain all edges of $C_n^m$).
By Lemma \ref{core}, either of three sets in (\ref{setdec}) is nonempty if and only if $m \mid d$.
So we assume $m \mid d$ in the following discussion.

For each $H \in \mathcal{V}_d(U^m;[\hat{C}_n^m])$, $\underline{H}$ is a hypertree, and by Lemma \ref{tree_root}, $H$ is uniquely rooted and $H=\hat{T}^m(\omg)$ for some tree $\hat{T}$ contained in $U$ with $\omg(\hat{T})=d/m$.
So  by Corollary \ref{TrV}, Eqs. (\ref{SpeT}) and (\ref{cdom}),
\begin{equation}\label{TraSum1}
\begin{split}
 \Tr_d(U^m;[\hat{C}_n^m])&:=d(m-1)^{|V(U^m)|} \sum_{H \in \mathcal{V}_d(U^m;[\hat{C}_n^m])} C_H N_{U^m}(\underline{H})\\
 &=
 \sum_{k=1}^{d/m} d(m-1)^{(m-1)(|E(U)|-k)-1}m^{k(m-2)} \sum_{\hat{T} \in \mathbf{T}_k} \tilde{c}_{d/m}(\hat{T}) N_{U}(\hat{T}).
 \end{split}
 \end{equation}

Suppose that $U$ is obtained from $C_n$ by attaching rooted trees $T_1,\ldots,T_n$ with their roots identified with $v_1,\ldots,v_n$ of the cycle respectively, where some trees $T_i$ may be trivial containing only one vertex (namely $v_i$).
Here we stress that each $T_i$ is a rooted tree with root $v_i$ for $i \in [n]$.
Note an isomorphism between two rooted trees is one preserving the roots.
If $\hat{T}$ is a rooted tree with root $v$, then $N_{T_i}(\hat{T})$ denotes the number of subgraphs of $T_i$ with root $v_i$ that is isomorphic to $\hat{T}$ mapping $v_i$ to $v$.

For each $H \in  \mathcal{V}_d(U^m;[C_n^m])$, $H|_{T_i^m}$ (if nonempty) for $i \in [n]$ and $H|_{C_n^m}$ are all Vebelen hypergraphs.
 By Lemma \ref{tree_root}, $H|_{T_i^m}=\hat{T}_i^m(\omg^i)$ for some subtree $\hat{T}_i$ of $T_i$ containing the vertex $v_i$ such that each edge $e$ of $\hat{T}_i^m$ repeat $m \omg^i(e)$ times in $H|_{T_i^m}$.
 Similarly, by Lemma \ref{core}, $H|_{C_n^m}=C_n^m(\omg^0)$ such that each edge $e$ of $C_n^m$ repeat $m \omg^0(e)$ times in $H|_{C_n^m}$.
 Surely, the number of edges of $H|_{T_i^m}$ for $i \in [n]$ and the number of edges of $H|_{C_n^m}$ are multiples of $m$.
 We also note $\mathcal{V}_d(U^m;[C_n^m]) \ne \emptyset$ only if $d/m \ge n$, with equality only if $\mathcal{V}_d(U^m;[C_n^m])=\{C_n^m(\omg^0)\}$ with $\omg^0(e)=1$ for each edge $e \in E(C_n^m)$.


By the discussion before, we have
\begin{equation}\label{Dec2} \mathcal{V}_d(U^m;[C_n^m])
=\cup_{d_0+\cdots+d_n=d, d_0/m \ge n \atop m \mid d_i, i\in \{0,1,\ldots,n\}}\cup_{0 \le s_i \le d_i/m}\mathcal{V}_{d_0,d_1,\ldots,d_n}^{s_1,\ldots,s_n}(U^m),
\end{equation}
where $\mathcal{V}_{d_0,d_1,\ldots,d_n}^{s_1,\ldots,s_n}(U^m)$ is the set of Veblen hypergraphs contains $d_0>0$ edges of $C_n^m$, and $d_i$ edges of $T_i^m$ among of which $m s_i$ edges contains the vertex $v_i$ for $i\in [n]$.
Furthermore,
$$\mathcal{V}_{d_0,d_1,\ldots,d_n}^{s_1,\ldots,s_n}(U^m)
=\cup\{C_n^m(\omg^0)\cup \cup_{i=1}^n\hat{T}_i^m(\omg^i): C_n^m(\omg^0) \in \mathcal{V}_{d_0}(C_n^m;[C_n^m]), \hat{T}_i^m(\omg^i) \in \mathcal{V}_{d_i;s_i}(T_i^m;[v_i]),i \in [n]\},$$
where $\mathcal{V}_{d_i;s_i}(T_i^m;[v_i])=\{\hat{T}_i^m(\omg^i) \in \mathcal{V}_{d_i}(T_i^m):  d_{v_i}(\hat{T}_i^m(\omg^i))=s_i,\hat{T}_i \subseteq T_i\}$, and $\hat{T}_i$ is a rooted tree with root $v_i$.
Note that in the above decomposition, some $d_i$'s are zeros, and $d_i=0$ if and only if $s_i=0$.


Denote $N=(m-1)|E(U)|$ the number of vertices of $U^m$, $N_i=(m-1)|E(T_i)|+1$ the number of vertices of $T_i^m$ for $i \in [n]$, and $N_0=(m-1)n$ the number of vertices of $C_n^m$.
Then $N=N_0+\sum_{i \in [n]}N_i - n$.
Let $H=C_n^m(\omg^0)\cup \cup_{i \in I(s)}\hat{T}_i^m(\omg^i) \in \mathcal{V}_{d_0,d_1,\ldots,d_n}^{s_1,\ldots,s_n}(U^m)$, where $I(s)=\{ i \in [n]: s_i>0\}$.
For each $F \in \mathcal{R}(H)$, $F_0:=F|_{C_n^m(\omg^0)}$ is an Euler rooting of $C_n^m(\omg^0)$, and $d_{v_i}^+(F_0)=(\omg^0(e^m_{i-1})+\omg^0(e^m_i))(m-1)=:t_i(m-1)$ for $i \in [n]$;
$F_i:=F|_{\hat{T}_i^m(\omg^i)}$ is also an Euler rooting of $\hat{T}_i^m(\omg^i)$, and $d_{v_i}^+(F_i)=s_i(m-1)$ for $i \in I(s)$.
We  have $\tau(F)=\tau(F_0) \prod_{i \in I(s)}\tau(F_i)$ and
$N_{U^m}(\underline{H})=\prod_{i \in I(s)}N_{T_i}(\hat{T}_i)$.
Denote
$$\Tr_{d_i;s_i}(T_i^m;[v_i])=d(m-1)^{N_i} \sum_{H \in \mathcal{V}_{d_i;s_i}(T_i^m;[v_i])} C_H N_{T_i^m}(\underline{H}), i \in I(s).$$
By Lemma \ref{TrCycle}, we have
\begin{equation}\label{Trds}
\begin{split}
\Tr_{d_0,d_1,\ldots,d_n}^{s_1,\ldots,s_n}(U^m)&:= d(m-1)^N \sum_{H \in \mathcal{V}_{d_0,d_1,\ldots,d_n}^{s_1,\ldots,s_n}(U^m)} C_H N_{U^m}(\underline{H})\\
&=d(m-1)^N \sum_{H \in \mathcal{V}_{d_0,d_1,\ldots,d_n}^{s_1,\ldots,s_n}(U^m)} \left(\sum_{F \in \mathcal{R}(H)} \frac{\tau(F)}{\prod_{v \in V(F)} d_v^+(F)} \right) N_{U^m}(\underline{H})\\
& = d(m-1)^N \sum_{C_n^m(\omg^0) \in \mathcal{V}_{d_0}(C_n^m;[C_n^m]),\atop
 \hat{T}_i^m(\omg^i) \in \mathcal{V}_{d_i;s_i}(T_i^m;[v_i]), i \in I(s)} \prod_{i \in I(s)} {s_i +t_i \choose s_i}\\
 & \cdot \sum_{F \in \mathcal{R}(\bigcup\limits_{i \in I(s)}\hat{T}_i^m(\omg^i))} \prod_{i \in I(s)}\frac{s_it_i(m-1)}{s_i+t_i} \frac{\tau(F_i)}{\prod_{v \in V(F_i)}
 d_v^+(F_i)} N_{T_i}(\hat{T}_i) \cdot \sum_{F_0 \in \mathcal{R}(C_n^m(\omg_0))}  \frac{\tau(F_0)}{\prod_{v \in V(F_0)} d_v^+(F_0)} \\
&=\frac{d(m-1)^{N-N_0-\sum_{i \in I(s)}(N_i-1)}}{d_0}
\prod_{i \in I(s)} \frac{s_i}{d_i}\sum_{T_i^m(\omg^i) \in \mathcal{V}_{d_i;t_i}(T_i^m;[v_i])}\sum_{F \in \mathcal{R}(\hat{T}_i^m(\omg^i))}
\frac{d_i(m-1)^{N_i} \tau(F_i)}{\prod_{v \in V(F_i)}d_v^+(F_i)}N_{T_i}(\hat{T}_i)\\
&  \cdot  \sum_{C_n^m(\omg_0)\in \mathcal{V}_{d_0}(C_n^m;[C_n^m])}\prod_{i \in I(s)} \frac{t_i}{s_i+t_i}{s_i +t_i \choose t_i} \sum_{F \in \mathcal{R}(C_n^m(\omg^0))}
\frac{d_0(m-1)^{N_0} \tau(F_0)}{\prod_{v \in V(F_0)}d_v^+(F_0)}
\\
&=\frac{d(m-1)^{N-N_0-\sum_{i \in I(s)}(N_i-1)}}{d_0}\prod_{i \in I(s)} \frac{s_i}{d_i} \Tr_{d_i;s_i}(T_i^m;[v_i])  \\
&  \cdot \sum_{\omg^0: \omg^0(C_n^m)=d_0/m}  \prod_{i=1}^n {\omg^0_{i-1}+\omg^0_{i} +s_i-1 \choose s_i} \cdot 2d_0 m^{n(m-2)-1} f_{C_n}(\omg_0).
\end{split}
\end{equation}

By a similar discussion as in Section \ref{TrTree}, we have a formula for $ \Tr_{d;s}(T^m;[v])$.
 Denote by $T(k;[v])$ the set of isomorphic classes of subtrees of $T$ with root $v$ and $k$ edges, and for each $\hat{T} \in T(k;[v])$ denote $\tilde{c}_{\ell;s}(\hat{T})$ as follows:
$$ \tilde{c}_{\ell;s}(\hat{T})= \sum_{\omg: \omg(\hat{T})=\ell, d_{v}(\hat{T}(\omg))=s } \prod_{e \in E(\hat{T})}\omg(e) \prod_{u \in V(\hat{T})} \frac{(d_u(\hat{T}(\omega))-1)!}{r_u(\hat{T}(\omega))}.$$
We have
\begin{equation}\label{Trdt}
\Tr_{d;s}(T^m;[v]) =\sum_{k=1}^{d/m} d(m-1)^{(m-1)(|E(T)|-k)}m^{k(m-2)} \sum_{\hat{T} \in T(k;[v])} \tilde{c}_{d/m;s}(\hat{T}) N_{T}(\hat{T}).
\end{equation}
So
\begin{equation}\label{TrProd}
\begin{split}
\prod_{i \in I} \Tr_{d_i;s_i}(T_i^m;[v_i]) &= \prod_{i \in I} d_i \sum_{k_i\in [d_i/m], i \in I}
(m-1)^{(m-1)\sum_{i \in I}(|E(T_i)|-k_i)}m^{(m-2)\sum_{i \in I} k_i}\\
& \quad \quad \cdot
\prod_{\hat{T}_i \in T_i(k_i;[v_i]), i \in I} \tilde{c}_{d_i/m;s_i}(\hat{T}_i)N_{T_i}(\hat{T}_i).
\end{split}
\end{equation}

We now give a formula for the traces of the power of unicyclic graphs.

\begin{thm}\label{MainThm}
Let $U$ be a unicyclic graph which is obtained from a cycle $C_n$ by attaching rooted trees $T_1,\ldots,T_n$ with their roots identified with  the vertices $v_1,\ldots,v_n$ of the cycle respectively.
For each $d$ with $m \mid d$,
\begin{equation}\label{Main1}
\begin{split}
\Tr_d(U^m)&=\sum_{k=1}^{d/m} d(m-1)^{(m-1)(|E(U)|-k)-1}m^{k(m-2)} \sum_{\hat{T} \in \mathbf{T}_k} \tilde{c}_{d/m}(\hat{T}) N_{U}(\hat{T})\\
& \quad + 2 m^{n(m-2)-1} \sum_{d_0+\cdots+d_n=d, d_0/m \ge n \atop m \mid d_i, i\in \{0,1,\ldots,n\}} \sum_{0 \le s_i \le d_i, i \in [n]} d(m-1)^{(m-1)(|E(U)|-n-\sum_{i \in I(s)}|E(T_i)|)} \prod_{i \in I(s)} s_i\\
& \quad \cdot \sum_{k_i\in [d_i/m], i \in I(s)}
(m-1)^{(m-1)\sum_{i \in I}(|E(T_i)|-k_i)}m^{(m-2)\sum_{i \in I(s)} k_i}
\prod_{\hat{T}_i \in T_i(k_i;[v_i]), i \in I(s)} \tilde{c}_{d_i/m;s_i}(\hat{T}_i)N_{T_i}(\hat{T}_i)\\
& \quad \cdot \sum_{\omg^0: \omg^0(C_n^m)=d_0/m}  \prod_{i=1}^n {\omg^0_{i-1}+\omg^0_{i} +s_i-1 \choose s_i} f_{C_n}(\omg_0).
\end{split}
\end{equation}
\end{thm}

\begin{proof}
By the decomposition (\ref{Dec2}), Eq. (\ref{Trds}) and Eq. (\ref{TrProd}),  we have
 \begin{equation}\label{TraSum2}
\begin{split}
\Tr_d(U^m;[C_n^m])&:=d(m-1)^{|V(U^m)|} \sum_{H \in \mathcal{V}_d(U^m;[C_n^m])} C_H N_{U^m}(\underline{H})\\
&=\sum_{d_0+\cdots+d_n=d, d_0/m \ge n \atop m \mid d_i, i\in \{0,1,\ldots,n\}} \sum_{0 \le s_i \le d_i, i \in [n]}\Tr_{d_0,d_1,\ldots,d_n}^{s_1,\ldots,s_n}(U^m)\\
&= \sum_{d_0+\cdots+d_n=d, d_0/m \ge n \atop m \mid d_i, i\in \{0,1,\ldots,n\}} \sum_{0 \le s_i \le d_i, i \in [n]} \frac{d(m-1)^{N-N_0-\sum_{i \in I(s)}(N_i-1)}}{d_0} \prod_{i \in I(s)} \frac{s_i}{d_i} \Tr_{d_i;s_i}(T_i^m;[v_i]) \\
& \quad \cdot \sum_{\omg^0: \omg^0(C_n^m)=d_0/m}  \prod_{i=1}^n {\omg^0_{i-1}+\omg^0_{i} +s_i-1 \choose s_i} 2d_0 m^{n(m-2)-1} f_{C_n}(\omg_0)\\
&= 2 m^{n(m-2)-1} \sum_{d_0+\cdots+d_n=d, d_0/m \ge n \atop m \mid d_i, i\in \{0,1,\ldots,n\}} \sum_{0 \le s_i \le d_i, i \in [n]} d(m-1)^{N-N_0-\sum_{i \in I(s)}(N_i-1)} \prod_{i \in I(s)} s_i\\
& \quad \cdot \sum_{k_i\in [d_i/m], i \in I(s)}
(m-1)^{(m-1)\sum_{i \in I}(|E(T_i)|-k_i)}m^{(m-2)\sum_{i \in I(s)} k_i}
\prod_{\hat{T}_i \in T_i(k_i;[v_i]), i \in I(s)} \tilde{c}_{d_i/m;s_i}(\hat{T}_i)N_{T_i}(\hat{T}_i)\\
& \quad \cdot \sum_{\omg^0: \omg^0(C_n^m)=d_0/m}  \prod_{i=1}^n {\omg^0_{i-1}+\omg^0_{i} +s_i-1 \choose s_i} f_{C_n}(\omg_0),
\end{split}
\end{equation}
 where $N=(m-1)|E(U)|$, $N_0=(m-1)n$ and $N_i=(m-1)|E(T_i)|+1$ are respectively the number of vertices of $U^m$, $C_n^m$ and $T_i^m$ for $i \in [n]$.

By the decomposition (\ref{setdec}), we have
$$  \Tr_d(U^m)=\Tr_d(U^m;[\hat{C}_n^m])+\Tr_d(U^m;[C_n^m])$$
and get the desired result by combing Eq. (\ref{TraSum1}) and Eq. (\ref{TraSum2}).
\end{proof}


\begin{cor}\label{dmnl}
Let $U$ be the unicyclic graph as defined in Theorem \ref{MainThm}. If $d/m=n$, then
\begin{equation}\label{main2}
\begin{split}
\Tr_{d}(U^m)& =\sum_{k=1}^n n(m-1)^{(m-1)(|E(U)|-k)-1}m^{k(m-2)+1} \sum_{\hat{T} \in \mathbf{T}_k} \tilde{c}_{n}(\hat{T}) N_{U}(\hat{T})\\
& \quad + 2n (n+1) (m-1)^{(m-1)(|E(U)|-n)}m^{n(m-2)}
\end{split}
\end{equation}
If $d/m<n$, then
\begin{equation}\label{main3}
\Tr_{d}(U^m) =\sum_{k=1}^{d/m} d(m-1)^{(m-1)(|E(U)|-k)-1}m^{k(m-2)} \sum_{\hat{T} \in \mathbf{T}_k} \tilde{c}_{d/m}(\hat{T}) N_{U}(\hat{T}).
\end{equation}
\end{cor}

\begin{proof}
If $d/m=n$, then $d_0=d=mn$, $d_i=s_i=0$ for $i \in [n]$, and $I(s)=\emptyset$ in Eq. (\ref{Main1}).
Also in this case, $\omg^0_i=1$ for $i \in [n]$, and $f_{C_n}(\omg^0)=n+1$. So we have
\begin{align*}
\Tr_{nm}(U^m)&=\sum_{k=1}^{d/m} d (m-1)^{(m-1)(|E(U)|-k)-1}m^{k(m-2)} \sum_{\hat{T} \in \mathbf{T}_k} \tilde{c}_{d/m}(\hat{T}) N_{U}(\hat{T})\\
& \quad + 2d(m-1)^{(m-1)(|E(U)|-n)}m^{n(m-2)-1}(n+1)\\
& =\sum_{k=1}^n n(m-1)^{(m-1)(|E(U)|-k)-1}m^{k(m-2)+1} \sum_{\hat{T} \in \mathbf{T}_k} \tilde{c}_{n}(\hat{T}) N_{U}(\hat{T})\\
& \quad + 2n (n+1) (m-1)^{(m-1)(|E(U)|-n)}m^{n(m-2)}.
\end{align*}

If $d/m<n$, then the second summand in (\ref{Main1}) does not appear.
The result follows.
\end{proof}

Recall the \emph{girth} of a graph $G$, denoted by $g(G)$, is the minimum length of the cycles of $G$. If $G$ contains no cycles, then we define $g(G) =+\infty$.
 At the end of this section, we will consider some special traces of a general graph with girth $g$.

\begin{thm}\label{GenMain}
Let $G$ be a graph with $n$ vertices and $p$ edges.
Then for each $m \ge 3$ and each $d$ with $m \mid d$ and $d/m < g(G)$,
\begin{equation}\label{main4}
\Tr_{d}(G^m) =\sum_{k=1}^{d/m} d(m-1)^{n-1-p+(m-1)(p-k)}m^{k(m-2)} \sum_{\hat{T} \in \mathbf{T}_k} \tilde{c}_{d/m}(\hat{T}) N_{G}(\hat{T}).
\end{equation}
\end{thm}

\begin{proof}
Observe that $G^m$ has $N:=n+p(m-2)$ vertices.
As $m \ge 3$, each edge of $G^m$ contains a cored vertex.
By Lemma \ref{core}, each edge of $H \in \mathcal{V}_d(G^m)$ repeats in a multiple of $m$ times, which implying that $d$ is multiple of $m$ if $\mathcal{V}_d(G^m) \ne \emptyset$.
 Also $H$ is a weighted graph, namely $H=\underline{H}(\omg)$ such that each edge $e$ of $\underline{H}$ repeats $m \omg(e)$ times in $H$, where $\omg: E(\underline{H}) \to \mathbb{Z}^+$.
 So, $\omg(\underline{H}):=\sum_{e \in E(\underline{H})}\omg(e)=d/m$, which implies that $\underline{H}$ has at most $d/m$ edges.

Now assume that $m \mid d$ and $d/m <g(G)$.
Then $\underline{H}$ contains no cycles, and hence $\underline{H}=\hat{T}^m$ for some tree $\hat{T} \subseteq G$.
So $$\mathcal{V}_d(G^m)=\{\hat{T}^m(\omg): \hat{T} \subseteq G, \omg(\hat{T}^m)=d/m\}.$$
By Corollary \ref{TrV}, Eq. (\ref{cdom}) and a similar discussion as in Eq. (\ref{Bu1}) by replacing $\T$ with $G^m$, $\hat{\T}$ with $\hat{T}^m$,
 \begin{align*}
 \Tr_d(G^m)& =d(m-1)^N \sum_{H \in \mathcal{V}_d(G^m)} C_H N_{ G^m}(\underline{H})\\
  &=\sum_{\hat{T}^m \in   \mathbf{T}^m_{\le d/m}}
 d(m-1)^{N-|V(\hat{T}^m|)} m^{(m-2)|E(\hat{T}^m)|}c_{d,m}(\hat{T}^m) N_{G^m}(\hat{T}^m)\\
 & =\sum_{k=1}^{d/m} d(m-1)^{n-1-p+(m-1)(p-k)} m^{k(m-2)}  \sum_{\hat{T}\in \mathbf{T}_k} \tilde{c}_{d/m}(\hat{T})N_{G}(\hat{T}).
 \end{align*}
\end{proof}

We note that if $G$ is a tree $T$ or unicyclic graph $U$, then
Eq. (\ref{main4}) becomes Eq. (\ref{Bu2}) or Eq. (\ref{main3}).

\section{Spectral characterization of unicyclic graphs}\label{SCUG}
In this section, we give some high-ordered cospectral invariants of graphs, in particular for unicylic graphs.
As applications, we prove that a class of unicylic graphs which are not DS but DHS.
We give two examples of infinitely many pairs of cospectral unicyclic graphs with different high-ordered spectra.
Denote by $P_n$ a path on $n$ vertices (as a graph).

\subsection{General graphs}
It is known that if $G_1$ is cospectral with $G_2$, then they have the same number of vertices and edges, respectively, namely
$$N_{G_1}(P_1)=N_{G_2}(P_1), N_{G_1}(P_2)=N_{G_2}(P_2).$$
We will prove that $N_G(P_3)$ is a high-ordered cospectral invariant of $G$.

\begin{lem}\label{NGen}
If $G_1$ is high-ordered cospectral with $G_2$, then
\begin{equation}\label{Gcl}\sum_{\hat{T} \in \mathbf{T}_k}\tilde{c}_{\ell}(\hat{T}) N_{G_1}(\hat{T})=\sum_{\hat{T} \in \mathbf{T}_k}\tilde{c}_{\ell}(\hat{T}) N_{G_2}(\hat{T})
\end{equation}
for all $\ell,k$ with $1 \le k \le \ell <\min\{g(G_1),g(G_2)\}$.
\end{lem}

\begin{proof}
As $\Tr_d(G_1^m)=\Tr_d(G_2^m)$, by Theorem \ref{GenMain}, for $m \mid d$ and $d/m < \min\{g(G_1),g(G_2)\}$,
$$\sum_{k=1}^{d/m} (m-1)^{n-1-p+(m-1)(p-k)}m^{k(m-2)} \sum_{\hat{T} \in \mathbf{T}_k} \tilde{c}_{d/m}(\hat{T})( N_{G_1}(\hat{T})-N_{G_2}(\hat{T}))=0,$$
where $n,p$ are respectively the number of vertices and edges of $G_1$ or $G_2$.
Let $d/m=\ell$, $g(m,k)=(m-1)^{n-1-p+(m-1)(p-k)}m^{k(m-2)}$, and $y(k)=\sum_{\hat{T} \in \mathbf{T}_k} \tilde{c}_{\ell}(\hat{T})( N_{U_2}(\hat{T})-N_{U_1}(\hat{T}))$.
Then
we have $$\sum_{k=1}^{\ell} g(m,k)y(k)=0.$$
As $g(m,k)=g(m,k-1)\frac{m^{m-2}}{(m-1)^{m-1}}=:g(m,k-1)\alpha(m)$ for $k >1$, $g(m,k)=g(m,1)\alpha(m)^{k-1}$.
Now taking $m$ as $\ell$ different integers $m_1,\ldots,m_\ell$ that are greater than $2$, we have
\begin{equation}\label{Vand1}
\left[\begin{array}{ccccc}
1 & \alpha(m_1) & \alpha(m_1)^2 & \cdots & \alpha(m_1)^{\ell-1}\\
1 & \alpha(m_2) & \alpha(m_2)^2 & \cdots & \alpha(m_2)^{\ell-1}\\
\vdots & \vdots& \vdots & \cdots & \vdots \\
1 & \alpha(m_\ell) & \alpha(m_2)^2 & \cdots & \alpha(m_\ell)^{\ell-1}
\end{array}\right]
\left[\begin{array}{c}
y(1)\\
y(2)\\
\vdots\\
y(\ell)
\end{array}\right]=0.
\end{equation}
As the matrix in Eq. (\ref{Vand1}) is nonsingular, we have $y(k)=0$ for $k \in [\ell]$.
\end{proof}

\begin{rmk}\label{NUn0}
In Lemma \ref{NGen}, suppose $\mathbf{T}_k$ has exactly $t_k$ trees $\hat{T}_1,\ldots,\hat{T}_{t_k}$.
If we can take $\ell$ as $t_k$ different integers $\ell_1, \ldots, \ell_{t_k}$ such that
$\min\{ \ell_i: i \in [t_k]\} \ge k$ and
$\max\{ \ell_i: i \in [t_k]\} < \min\{g(G_1),g(G_2)\}$,
then by Lemma \ref{NGen}, we have
\begin{equation}\label{Mat0}
\left[\begin{array}{ccccc}
\tilde{c}_{\ell_1}(\hat{T}_1) & \tilde{c}_{\ell_1}(\hat{T}_2) &  \cdots & \tilde{c}_{\ell_1}(\hat{T}_{t_k})\\
\tilde{c}_{\ell_2}(\hat{T}_1) & \tilde{c}_{\ell_2}(\hat{T}_2) &  \cdots & \tilde{c}_{\ell_2}(\hat{T}_{t_k})\\
\vdots & \vdots&  \cdots & \vdots \\
\tilde{c}_{\ell_{t_k}}(\hat{T}_1) & \tilde{c}_{\ell_{t_k}}(\hat{T}_2) &  \cdots & \tilde{c}_{\ell_{t_k}}(\hat{T}_{t_k})\\
\end{array}\right]
\left[\begin{array}{c}
z(1)\\
z(2)\\
\vdots\\
z(t_k)
\end{array}\right]=0,
\end{equation}
where  $z(i)=N_{G_1}(\hat{T}_i)-N_{G_2}(\hat{T}_i)$ for $i \in [t_k]$.
If the matrix in (\ref{Mat0}), denoted by $\tilde{C}$, is nonsingular, then the equation in (\ref{Mat0}) has only zero solution, which implies that
for all  $\hat{T} \in \mathbf{T}_k$,
$$N_{G_1}(\hat{T})=N_{G_2}(\hat{T}).$$

Obviously, $\mathbf{T}_2=\{P_3\}$.
By definition, $\tilde{c}_{2}(P_3)>0$.
As any graph has girth greater than $2$,
if $G_1$ is high-order cospectral with $G_2$, then
\begin{equation}\label{T2} N_{G_1}(P_3)=N_{G_2}(P_3).\end{equation}

Chen, Sun and Bu \cite{CSBu} defined a parameter $c_{2\ell}(\hat{T})$ and a matrix $C=(c_{2\ell_i}(\hat{T}_j))$ of size $t_k$ similar to $\tilde{C}$.
By definition, we find that
\begin{equation}\label{tcc}\tilde{c}_\ell(\hat{T})=\frac{c_{2\ell}(\hat{T})}{2 \ell},
\end{equation}
which implies that
$$ \tilde{C}=\text{diag}\left\{\frac{1}{2 \ell_1},\ldots,\frac{1}{2 \ell_{t_k}}\right\} C.$$
So $\tilde{C}$ and $C$ have the same singularity.

\end{rmk}

\begin{cor}\label{Main2}
If $G_1$ is high-ordered cospectral with $G_2$, then
$$ N_{G_1}(P_3)=N_{G_2}(P_3).$$
\end{cor}

\begin{cor}\label{ID4}
If $G_1$ is high-ordered cospectral with $G_2$, then
$$ N_{G_1}(P_k)=N_{G_2}(P_k), k \in \{1,2,3\},$$
 and
 $$N_{G_1}(C_3)=N_{G_2}(C_3), N_{G_1}(C_4)=N_{G_2}(C_4).$$
\end{cor}

\begin{proof}
It is known for a graph $G$, $\Tr_3(G)=6N_G(C_3)$ and
$\Tr_4(G)=2N_G(P_2)+4N_G(P_3)+8N_G(C_4)$ (\cite{Cve}).
The result follows by Corollary \ref{Main2} and the above equalities.
\end{proof}

\begin{cor}\label{CosG0}
Let $G_1$ and $G_2$ be two graphs with girth greater than or equal to $g$.
If $G_1$ is high-order cospectral with $G_2$, then

\begin{itemize}

 \item[(1)] if $g=5$, then $N_{G_1}(\hat{T})=N_{G_2}(\hat{T})$ for all trees $\hat{T}$ with at most $3$ edges.

  \item[(2)] if $g=7$, then $N_{G_1}(\hat{T})=N_{G_2}(\hat{T})$ for all trees $\hat{T}$ with at most $4$ edges.

  \item[(3)] if $g=11$, then $N_{G_1}(\hat{T})=N_{G_2}(\hat{T})$ for all trees $\hat{T}$ with at most $5$ edges.
  \end{itemize}
\end{cor}

\begin{proof}

By Tables 1 and 2 in \cite{CSBu} and Remark \ref{NUn0} , $\mathbf{T}_3$ has $2$ trees, and the matrix in (\ref{Mat0}) is nonsingular by taking $\ell \in \{3,4\}$.
Similarly, $\mathbf{T}_4$ has $3$ trees, and the matrix in (\ref{Mat0}) is nonsingular by taking $\ell \in \{4,5,6\}$;
$\mathbf{T}_5$ has $6$ trees, and the matrix in (\ref{Mat0}) is nonsingular by taking $\ell \in \{5,6,\ldots,10\}$.
The result follows by Remark \ref{NUn0}.
\end{proof}

\subsection{Unicylic graphs}
We first prove that the girth is a high-ordered cospectral invariant of unicyclic graphs.

\begin{lem}\label{girth}
Let $U_1$ (respectively, $U_2$) be a unicyclic graph containing a cycle $C_{n_1}$ (respectively, $C_{n_2}$).
If $U_1$ is high-order cospectral with $U_2$, then $n_1=n_2$.
\end{lem}

\begin{proof}
Since $U_1$ is high-order cospectral with $U_2$, we have
$\Tr_d(U_1^m)=\Tr_d(U_2^m)$ for all $d$ and $m \ge 3$.
Suppose that $n_1 < n_2$.
Now taking $d=n_1m$, by Corollary  \ref{dmnl} we have
\begin{align*}
& \sum_{k=1}^{n_1} n_1(m-1)^{(m-1)(|E(U)|-k)-1}m^{k(m-2)+1} \sum_{\hat{T} \in \mathbf{T}_k} \tilde{c}_{n_1}(\hat{T}) (N_{U_2}(\hat{T})-N_{U_1}(\hat{T}))\\
& = 2n_1 (n_1+1) (m-1)^{(m-1)(|E(U)|-n_1)}m^{n_1(m-2)}.
\end{align*}
By a simple calculation, we have
 $$\sum_{k=1}^{n_1} (m-1)^{(m-1)(n_1-k)-1}m^{(k-n_1)(m-2)+1} \sum_{\hat{T} \in \mathbf{T}_k} \tilde{c}_{n_1}(\hat{T}) (N_{U_2}(\hat{T})-N_{U_1}(\hat{T}))
 = 2 (n_1+1).
$$
Let $h(m,k)= (m-1)^{(m-1)(n_1-k)-1}m^{(k-n_1)(m-2)+1}$, and
$y(k)=\sum_{\hat{T} \in \mathbf{T}_k} \tilde{c}_{n_1}(\hat{T})( N_{U_2}(\hat{T})-N_{U_1}(\hat{T}))$.
Then
$h(m,k)=h(m,1)\alpha(m)^{k-1}$, and
$$ \sum_{k=1}^{n_1} h(m,1) \alpha(m)^{k-1} y(k)=2(n_1+1),$$
where $\alpha(m)=\frac{m^{m-2}}{(m-1)^{m-1}}$.
Taking $m$ as $n_1$ different integers $m_1,\ldots,m_{n_1}$ that are greater than $2$, we have
\begin{equation}\label{Girth}
\left[\begin{array}{cccc}
h(m_1,1) & h(m_1,1)\alpha(m_1) &  \cdots & h(m_1,1)\alpha(m_1)^{n_1-1}\\
h(m_2,1) & h(m_2,1)\alpha(m_2) &  \cdots & h(m_2,1)\alpha(m_2)^{n_1-1}\\
\vdots & \vdots&  \cdots & \vdots \\
h(m_{n_1},1) & h(m_{n_1},1)\alpha(m_{n_1}) &  \cdots & h(m_{n_1},1)\alpha(m_{n_1})^{{n_1}-1}
\end{array}\right]
\left[\begin{array}{c}
y(1)\\
y(2)\\
\vdots\\
y(n_1)
\end{array}\right]=\left[\begin{array}{c}
2(n_1+1)\\
2(n_1+1)\\
\vdots\\
2(n_1+1)
\end{array}\right].
\end{equation}
It is easily found the matrix $H:=(h(m_i,1)\alpha(m_i)^{j-1})$ in Eq. (\ref{Girth}) is nonsingular, and by Cramer's rule
\begin{equation}\label{Root} y(1)=(\det H)^{-1} \det M, \end{equation}
where
$$ M=\left[\begin{array}{cccc}
2(n_1+1) & h(m_1,1)\alpha(m_1) &  \cdots & h(m_1,1)\alpha(m_1)^{n_1-1}\\
2(n_1+1) & h(m_2,1)\alpha(m_2) &  \cdots & h(m_2,1)\alpha(m_2)^{n_1-1}\\
\vdots & \vdots&  \cdots & \vdots \\
2(n_1+1) & h(m_{n_1},1)\alpha(m_{n_1}) &  \cdots & h(m_{n_1},1)\alpha(m_{n_1})^{{n_1}-1}
\end{array}\right],$$
which is obtained from $H$ by replacing the first column by $(2(n_1+1), \ldots, 2(n_1+1))^\top$.

We will prove that there exists distinct integers $m_1,m_2,\ldots,m_{n_1}$ such that $\det M \ne 0$.
First given any $n_1-1$ distinct integers $m_1,m_2,\ldots,m_{n_1-1}$, the matrix $M(n_1,1)$ is nonsingular so that the equation
$$  (2(n_1+1), \ldots, 2(n_1+1))^\top=M(n_1,1) x$$
has a unique solution $x=(x_1,\ldots, x_{n_1-1})^\top$, where $M(n_1,1)$ is obtained from $M$ by deleting the last row and the first column.
Let $\ell$ be the smallest index such that $x_\ell \ne 0$.
Then $\ell \le n_1-2$; otherwise we have
$$  (2(n_1+1), \ldots, 2(n_1+1))^\top=
(h(m_1,1)\alpha(m_1)^{n_1-1},\ldots,h(m_{n_1-1},1)\alpha(m_{n_1-1})^{n_1-1})^\top\cdot x_{n_1-1}, $$
a contradiction.
Now consider the following equation in the variable $m$:
\begin{equation}\label{negequ}
2(n_1+1)=h(m,1)\alpha(m) x_1 + \cdots+h(m,1)\alpha(m)^{n_1-1} x_{n_1-1}.
\end{equation}
Let
$$ f(m)=h(m,1)\alpha(m) x_1 +\cdots+ h(m,1)\alpha(m)^{n_1-1} x_{n_1-1}
=h(m,2)x_1+ \cdots + h(m,n_1) x_{n_1-1}.$$
Note that $h(m,k) \sim e^{-(n_1-k)}(m-1)^{n_1-k}$ when $m \to +\infty$.
As $x_\ell \ne 0$ for some smallest $\ell \le n_1-2$,
$$f(m) \sim x_\ell h(m,\ell+1) \sim x_\ell e^{-(n_1-\ell- 1)} (m-1)^{n_1-\ell- 1}.$$
When $m \to +\infty$, then $|f(m)| \to +\infty $.
So, there exists a positive integer $m=m_0$ such that Eq. (\ref{negequ}) does not hold.
Surely, $m_0$ is not equal to any of $m_1,\ldots, m_{n_1-1}$ by the definition of $x$.
Taking $m_{n_1}=m_0$, $\det M \ne 0$ as the first column can not be a linear combination of other columns.

By Eq. (\ref{Root}), we get $y(1) \ne 0$.
However, by the definition of $y(k)$, $y(1)=0$ as $N_{U_2}(\hat{T})=N_{U_1}(\hat{T})$ when $\hat{T}=P_2$.
So we get the desired result, namely $n_1=n_2$.
\end{proof}

\begin{lem}\label{CosT}
Let $U_1$ (respectively, $U_2$) be a unicyclic graph containing a cycle $C_{n}$.
If $U_1$ is high-order cospectral with $U_2$, then
$$\sum_{\hat{T} \in \mathbf{T}_k}\tilde{c}_{\ell}(\hat{T}) N_{U_1}(\hat{T})=\sum_{\hat{T} \in \mathbf{T}_k}\tilde{c}_{\ell}(\hat{T}) N_{U_2}(\hat{T})$$
for all $\ell,k$ with $1 \le k \le \ell \le n$.
\end{lem}

\begin{proof}
As $\Tr_d(U_1^m)=\Tr_d(U_2^m)$, by Corollary  \ref{dmnl}, for $m \mid d$ and $d/m \le n$, we have
$$\sum_{k=1}^{d/m} (m-1)^{(m-1)(|E(U)|-k)-1}m^{k(m-2)} \sum_{\hat{T} \in \mathbf{T}_k} \tilde{c}_{d/m}(\hat{T})( N_{U_2}(\hat{T})-N_{U_1}(\hat{T}))=0.$$
Let $d/m=\ell$, $g(m,k)=(m-1)^{(m-1)(|E(U)|-k)-1}m^{k(m-2)}$, and $y(k)=\sum_{\hat{T} \in \mathbf{T}_k} \tilde{c}_{\ell}(\hat{T})( N_{U_2}(\hat{T})-N_{U_1}(\hat{T}))$.
Then
we have $$\sum_{k=1}^{\ell} g(m,k)y(k)=0.$$
By a similar discussion as in Lemma \ref{NGen}, we get the desired result.
\end{proof}

\begin{rmk}\label{NUn}
In Lemma \ref{CosT}, suppose $\mathbf{T}_k$ has exactly $t_k$ trees $\hat{T}_1,\ldots,\hat{T}_{t_k}$.
If we can take $\ell$ as $t_k$ different integers $\ell_1, \ldots, \ell_{t_k}$ such that
$\min\{ \ell_i: i \in [t_k]\} \ge k$ and
$\max\{ \ell_i: i \in [t_k]\} \le n$,
then by Lemma \ref{CosT}, we have the matrix equation in Eq. (\ref{Mat0}):
$\tilde{C}z=0$.
If the matrix $\tilde{C}$ is nonsingular, then $z$ has only zero solution, which implies that
for all  $\hat{T} \in \mathbf{T}_k$,
$$N_{U_1}(\hat{T})=N_{U_2}(\hat{T}).$$
%
%
\end{rmk}

\begin{cor}\label{CosG}
Let $U_1$ and $U_2$ be unicyclic graphs both with girth $g$.
If $U_1$ is high-order cospectral with $U_2$, then

\begin{itemize}

 \item[(1)] if $g \ge 4$, then $N_{U_1}(\hat{T})=N_{U_2}(\hat{T})$ for all trees $\hat{T}$ with at most $3$ edges.

  \item[(2)] if $g \ge 6$, then $N_{U_1}(\hat{T})=N_{U_2}(\hat{T})$ for all trees $\hat{T}$ with at most $4$ edges.

  \item[(3)] if $g \ge 10$, then $N_{U_1}(\hat{T})=N_{U_2}(\hat{T})$ for all trees $\hat{T}$ with at most $5$ edges.
  \end{itemize}
\end{cor}

\begin{proof}
The proof is similar to that of Corollary \ref{CosG0} by noting that the number $\ell$ in Lemma \ref{CosT} is allowed to equal the girth $g$.
We omit the details.
\end{proof}

\begin{figure}[h]
\centering
\includegraphics[scale=.8]{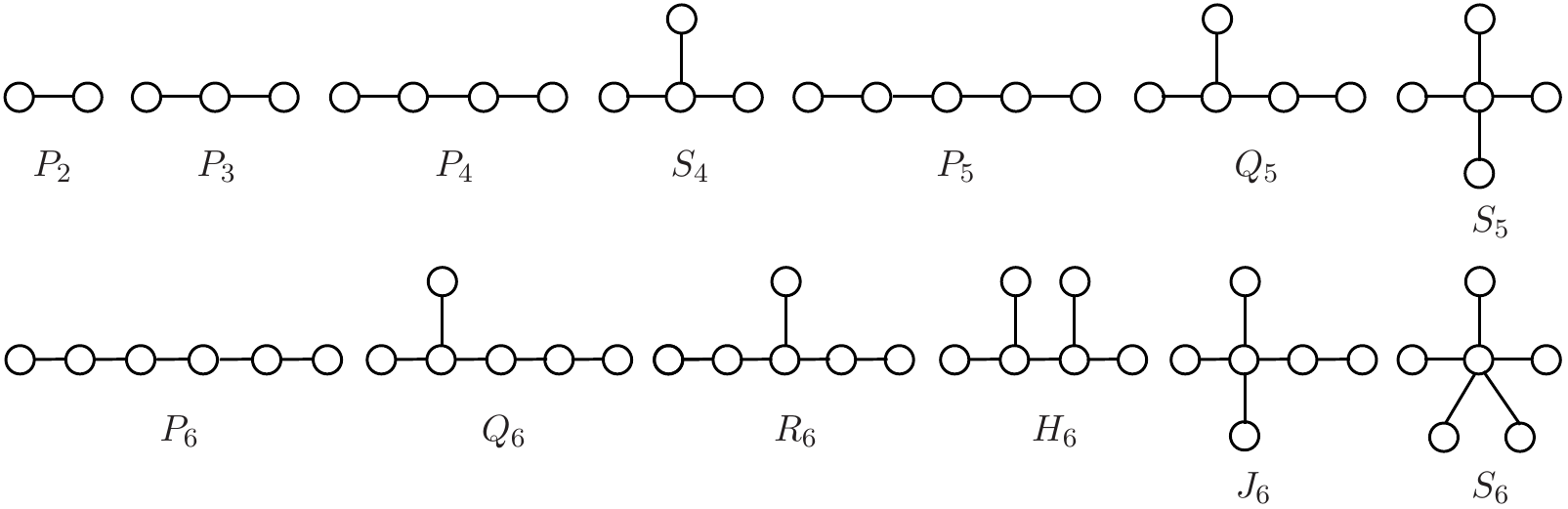}
\caption{The trees with at most $5$ edges}\label{tree5}
\end{figure}

%

We further investigate the high-order cospectral pairs of unicyclic graphs with girth at most $4$.

\begin{cor}\label{g34}
Let $U_1$ and $U_2$ be high-order cospectral unicyclic graphs with girth $g$.

\begin{itemize}

 \item[(1)]  If $g=3$, then
$$N_{U_1}(P_4)+2N_{U_1}(S_4)=N_{U_2}(P_4)+2N_{U_2}(S_4).$$

 \item[(2)]  If $g=4$, then
$$ N_{U_1}(P_5)+2N_{U_1}(Q_5)+6N_{U_1}(S_5)=N_{U_2}(P_5)+2N_{U_2}(Q_5)+6N_{U_2}(S_5).$$
\end{itemize}
\end{cor}

\begin{proof}
By Lemma \ref{CosT}, we have
\begin{equation}\label{Trg34}\sum_{\hat{T} \in \mathbf{T}_g}\tilde{c}_{g}(\hat{T}) N_{U_1}(\hat{T})=\sum_{\hat{T} \in \mathbf{T}_g}\tilde{c}_{g}(\hat{T}) N_{U_2}(\hat{T}).
\end{equation}
If $g=3$, $\mathbf{T}_3=\{P_4,S_4\}$ and $\tilde{c}_{3}(P_4)=1$, $\tilde{c}_{3}(S_4)=2$; and if $g=4$,
$\mathbf{T}_4=\{P_5,Q_5,S_5\}$ and
$\tilde{c}_{4}(P_5)=1$, $\tilde{c}_{4}(Q_5)=2$ and $\tilde{c}_{4}(S_5)=6$,
where the graphs $S_4,Q_5,S_5$ are listed in Fig. \ref{tree5}.
The result follows by substituting the above values into Eq. (\ref{Trg34}).
\end{proof}

\subsection{A class of DHS unicylic graphs}
By Lemma \ref{girth}, the girth is a high-ordered cospectral invariant of unicyclic graphs.
However, it does not hold for the usual cospectral invariant of unicyclic graphs.
Let $H(n; q, n_1, n_2)$ denotes the unicyclic graph on $n$ vertices which is obtained from a cycle $C_q$ by attaching two paths $P_{n_1+1}$ and $P_{n_2+1}$ at the same vertex of the cycle $C_q$, and $n_1,n_2 \ge 1$.
Liu et al. \cite{LiuW} proved the following cospectral mates: $H(12;6,1,5)$ and $H(12;8,2,2)$, $H(n;2a+6,a,a+2)$ and $\Lambda(a,a,2a+2)$,
$H(n;2b,b,b)$ and $\Theta(b-2,2b-3,b-1)$,
where $\Lambda(a,a,2a+2)$ is a unicyclic graph with girth $6$ and $a$ being positive even number, and  $\Theta(b-2,2b-3,b-1)$ is a bicyclic graph with $b$ being even number greater than $2$; see Theorem 3.4, Lemma 5.8 and Lemma 5.11 in \cite{LiuW}.
We now further investigate the high-ordered spectral property of $H(n; q, n_1, n_2)$.

\begin{thm}
The unicyclic graph $H(n; q, n_1, n_2)$  is DHS when $q \ge 5$.
\end{thm}

\begin{proof}
Let $H:=H(n; q, n_1, n_2)$.
Suppose that $G$ is a graph high-ordered cospectral with $H$, which has $n_1$ vertices of degree $1$, $n_2$ vertices of degree $2$ and $n_3$ vertices of degree greater than or equal to $3$.
By the relation $N_H(P_i)=N_G(P_i)$ for $i \in [3]$, and noting that $N_{H}(P_3)=n+3$, we have
\begin{equation}\label{P123}n_1+n_2+n_3=n, n_1+2n_2+\sum_{v: d_v \ge 3}d_v=2n, n_2+\sum_{v: d_v \ge 3}{d_v \choose 2} =n+3.
\end{equation}
By the 3rd and 2nd equalities,
$$ n+3 \ge n_2+\sum_{v: d_v \ge 3}d_v=2n-n_1-n_2,$$
which implies that $n_3 \le 3$.
The result will be arrived by the following discussion.

Case 1: $n_3 =3$. Then for each $v$ with $d_v \ge 3$, $d_v=3$.
So we get $n_1=3,n_2=n-6,n_3=3$.
By Corollary \ref{ID4}, $N_{H}(C_k)=N_G(C_k)$ for $k=3,4$.
As $H$ contains no $C_3$ or $C_4$, the girth of $G$ is at least $5$.
Observe that $N_G(S_4) =3$ and $N_H(S_4)=4$.
However, by Corollary \ref{CosG0}, $N_H(S_4)=N_G(S_4)$; a contradiction.
So this case cannot happen.

Case 2: $n_3=2$. Then $G$ contains exactly two vertices, say $u,v$, both with degree greater than or equal to $3$.
By the 1st and 2nd equalities in Eq. (\ref{P123}), we have
$d_u+d_v=n+2-n_2$.
If both $d_u$ and $d_v$ are  greater than or equal to $4$, then by the 3rd equality in Eq. (\ref{P123}),
$$n+3 \ge n_2+\frac{3}{2}(d_u+d_v) =n_2+\frac{3}{2}(n+2-n_2)=\frac{3}{2}n-\frac{1}{2}n_2+3,$$
which implies that $n \le n_2$; a contradiction.
So there exists a vertex among $u,v$, say $u$ with $d_u=3$.
Also by Eq. (\ref{P123}),
$$ d_v=n-n_2-1, {d_v \choose 2}=n-n_2.$$
If letting $x=n-n_2$, by the above equalities, we have
$$ x^2-5x+2=0,$$
which implies that $x$ is an irrational number; a contradiction.
So this case also cannot happen.

Case 3: $n_3=1$. Then $G$ have only one vertex say $u$ with degree greater than or equal to $3$.
By Eq. (\ref{P123}), we have
$$d_v=n-n_2+1, {d_v \choose 2}=n-n_2+3.$$
By a simple computation, we have $n_1=2,n_2=n-3,n_3=1$ and $d_v=4$.

Let $G_0$ be a component of $G$ containing the vertex $v$.
Then $G_0$ must contain a cycle; otherwise $G_0$ is a tree with at least $4$ vertices of degree $1$; a contradiction to $n_1=2$.
As $G_0$ contains exactly one vertex greater than $2$, namely the vertex $v$ with degree $4$, $G_0$ is of form $H(n; q', n'_1, n'_2)$.

We assert that $G$ has no other components other than $G_0$.
Otherwise, let $G_1$ be another component of $G$.
Then $G_1$ is a cycle with the largest eigenvalue $\la_1(G_1)=2$ as it only contains vertices with degree $2$.
We also note $\la_1(G_0) > 2$ as $G_0$ contains the cycle $C_q$ as a proper subgraph.
So the second eigenvalue $\la_2(G)$ of $G$ holds that
$$\la_2(G) \ge \min\{\la_1(G_0),\la_1(G_1)) \ge 2.$$
Let $\tilde{v}$ be the vertex of $H$ with degree $4$.
Note that $H-\tilde{v}$ consists of paths, and the largest eigenvalue of a path is less than $2$.
By the interlacing theorem,
$$ \la_2(H) \le \la_1(H-\tilde{v}) <2,$$
which yields a contradiction as $\la_2(G)=\la_2(H)$.
So $G$ is of form $H(n; q', n'_1, n'_2)$.
By Theorem 3.4 of \cite{LiuW}, except for $H(12;6,1,5)$ and $H(12;8,2,2)$,
no two non-isomorphic graphs of form $H(n';q',n'_1,n'_2)$ are cospectral.
As two high-ordered cospectral unicyclic graphs has the same girth by Lemma \ref{girth}, $G$ is isomorphic to $H(n;q, n_1,n_2)$.

Case 4: $n_3=0$. Then $G$ is a union of paths and/or cycles.
This is impossible as $\la_1(G) \le 2 < \la_1(H)$.

\end{proof}

\subsection{Cospectral unicyclic graph mates with different high-ordered spectra}

There are two methods to construct cospectral graphs.
The first is to use coalescence introduced by Schwenk \cite{Sch} (or see \cite{DamH1}).
Let $G$ be a graph with root $u$ and $H$ be a graph with root $v$.
The \emph{coalescence} of $G$ and $H$ is obtained by identifying the root $u$ of $G$ with the root $v$ of $H$, denoted by $G(u) \odot H(v)$.

\begin{lem}\label{coa}\cite{Sch, DamH1}
Let $G$ be a graph with root $u$ and $H$ be a graph with root $v$.
If $G$ is cospectral with $H$ and $G-u$ is cospectral with $H-v$,
then for any graph $\Gamma$ with root $w$,
$G(u) \odot \Gamma(w)$ is cospectral with $H(v) \odot \Gamma(w)$.
\end{lem}

The second method is to use rooted product introduced by Godsil and McKay \cite{GodM2} (or see \cite{Lep}).
Let $G$ be a graph and let $\Gamma$ be a rooted graph.
The \emph{root product} $G(\Gamma)$ is the graph obtained by identifying each vertex of $G$ with the root of a copy of $\Gamma$.

\begin{lem}[\cite{Lep}]\label{root}
If $G$ is cospectral with a  graph $H$, then for any rooted graph $\Gamma$,
$G(\Gamma)$ is cospectral with $H(\Gamma)$.
\end{lem}

\begin{figure}[h]
\centering
\includegraphics[scale=.8]{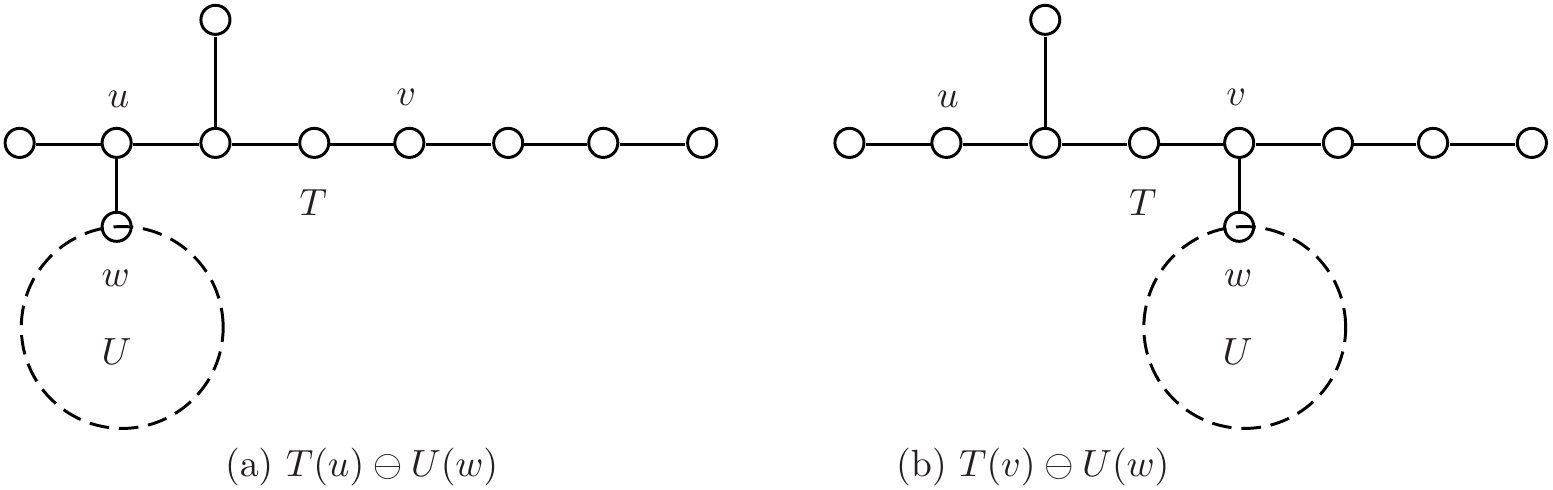}
\caption{The graphs $T(u)\ominus U(w)$ and $T(v)\ominus U(w)$}\label{coal}
\end{figure}

We now construct infinitely many pairs of cospectral unicyclic graphs
with different high-order spectrum.
Let $T$ be the tree in Fig. \ref{coal} (the graph with solid edges) with two specified vertices $u$ and $v$.
It was shown that $T-u$ and $T-v$ have the same spectrum \cite{Sch}.
Let $U$ be any unicyclic graph with root $w$.
Denote by $T(u)\ominus U(w)$ (respectively, $T(v)\ominus U(w)$) the graph obtained by adding an edge between $u$ and $w$ (respectively, between $v$ and $w$); see Fig. \ref{coal}.
By Lemma \ref{coa}, $T(u)\ominus U(w)$ is cospectral with $T(v)\ominus U(w)$.
We now prove that $T(u)\ominus U(w)$ is not high-ordered cospectral with $T(v)\ominus U(w)$.

\begin{cor}\label{ominus}
$T(u)\ominus U(w)$ is not high-ordered cospectral with $T(v)\ominus U(w)$.
\end{cor}

\begin{proof}

Let $G_{uw}:=T(u)\ominus U(w)$ and $G_{vw}:=T(v)\ominus U(w)$, and
let $C_n$ be the cycle contained in $U$.
We will compare $\Tr_{d}(G_{uw}^m)$ and $\Tr_{d}(G_{vw}^m)$ with $d=5m$.

Let $H \in \mathcal{V}_d(G_{uw}^m,[C_n^m])$.
By Lemma \ref{core}, $H$ is a weighted hypergraph $\underline{H}(\omg)$ which contains $C_n^m$  such that each edge $e$ of $\underline{H}$ repeats in $m \omg(e)$ times.
So $\omg(\underline{H})=5m/5=5$, which implies that $\underline{H}$ has at most $5$ edges.
As $C_n^m$ has at least $3$ edges, $\underline{H}$ has at most $2$ edges outside $U^m$.
Let $\tilde{G}_{uw}$ be the subgraph of $G_{uw}$ induced by the vertices of $U$, $u$ and its neighbors.
Then $\underline{H}$ is contained in $\tilde{G}_{uw}^m$.
So $\mathcal{V}_d(G_{uw}^m,[C_n^m])=\mathcal{V}_d(\tilde{G}_{uw}^m,[C_n^m])$.
Similarly, if letting $\tilde{G}_{vw}$ be the subgraph of $G_{vw}$ induced by the vertices of $U$, $v$ and its neighbors, then for each $H \in \mathcal{V}_d(G_{vw}^m,[C_n^m])$, $\underline{H}$ is contained in $\tilde{G}_{vw}^m$, and hence $\mathcal{V}_d(G_{vw}^m,[C_n^m])=\mathcal{V}_d(\tilde{G}_{vw}^m,[C_n^m])$.
Note that $\tilde{G}_{uw}$ is isomorphic to $\tilde{G}_{vw}$.
So $\mathcal{V}_d(G_{uw}^m,[C_n^m])$ is equal to $\mathcal{V}_d(G_{vw}^m,[C_n^m])$ under isomorphism.

By definition,
\begin{align*}
\Tr_d(G_{uw}^m;[C_n^m])&=d(m-1)^{|V(G_{uw}^m)|}\sum_{H \in \mathcal{V}_d(G_{uw}^m,[C_n^m])}C_H N_{G_{uw}^m}(\underline{H})\\
&=d(m-1)^{|V(G_{uw}^m)|}\sum_{H \in \mathcal{V}_d(\tilde{G}_{uw}^m,[C_n^m])}C_H N_{\tilde{G}_{uw}^m}(\underline{H})\\
&=d(m-1)^{|V(G_{vw}^m)|}\sum_{H \in \mathcal{V}_d(\tilde{G}_{vw}^m,[C_n^m])}C_H N_{\tilde{G}_{vw}^m}(\underline{H})\\
&=d(m-1)^{|V(G_{vw}^m)|}\sum_{H \in \mathcal{V}_d({G}_{vw}^m,[C_n^m])}C_H N_{{G}_{vw}^m}(\underline{H})\\
&=\Tr_d(G_{vw}^m;[C_n^m]).
\end{align*}

By Eq. (\ref{setdec}) and Eq. (\ref{TraSum1}), we have
\begin{align*}
\Tr_d(G_{uw}^m)-\Tr_d(G_{vw}^m)&=\Tr_d(G_{uw}^m;[\hat{C}_n^m])-\Tr_d(G_{vw}^m;[\hat{C}_n^m])\\
&=\sum_{k=1}^{5} d(m-1)^{(m-1)(p-k)-1}m^{k(m-2)} \sum_{\hat{T} \in \mathbf{T}_k} \tilde{c}_{5}(\hat{T}) (N_{G_{uw}}(\hat{T})-N_{G_{vw}}(\hat{T})),
\end{align*}
where $p$ denotes the number of edges of $G_{uw}$ or $G_{vw}$.
We find that for each tree $\hat{T} \in \mathbf{T}_{\le 5} \setminus \{P_5,Q_5,P_6,Q_6,H_6\}$, $N_{G_{uw}}(\hat{T})=N_{G_{vw}}(\hat{T})$,
where $Q_5,Q_6,H_6$ are listed in Fig. \ref{tree5}.
So
\begin{align*}
\Tr_d(G_{uw}^m)-\Tr_d(G_{vw}^m)&=
 dg(m,4)\big[\tilde{c}_{5}(P_5) (N_{G_{uw}}(P_5)-N_{G_{vw}}(P_5))+\tilde{c}_{5}(Q_5) (N_{G_{uw}}(Q_5)-N_{G_{vw}}(Q_5) \big]\\
 &\quad + dg(m,5)\big[\tilde{c}_{5}(P_6) (N_{G_{uw}}(P_6)-N_{G_{vw}}(P_6))+\tilde{c}_{5}(Q_6) (N_{G_{uw}}(Q_6)-N_{G_{vw}}(Q_6))\\
 & \quad \quad + \tilde{c}_{5}(H_6) (N_{G_{uw}}(H_6)-N_{G_{vw}}(H_6))\big],
\end{align*}
where $g(m,k)=(m-1)^{(m-1)(M-k)-1}m^{k(m-2)}$.
By a direct computation or referring Tables 1 and 2 in \cite{CSBu} with the relation (\ref{tcc}), we have
$$ \tilde{c}_{5}(P_5)=6, \tilde{c}_{5}(Q_5)=14, \tilde{c}_{5}(P_6)=1, \tilde{c}_{5}(Q_6)=2, \tilde{c}_{5}(H_6)=4.$$
We also have
$$N_{G_{uw}}(P_5)-N_{G_{vw}}(P_5)=-2, N_{G_{uw}}(Q_5)-N_{G_{vw}}(Q_5)=1,$$
$$
N_{G_{uw}}(P_6)-N_{G_{vw}}(P_6)=-2d_v, N_{G_{uw}}(Q_6)-N_{G_{vw}}(Q_6)=d_v-3,
N_{G_{uw}}(H_6)-N_{G_{vw}}(H_6)=1.$$
So,
$$\Tr_d(G_{uw}^m)-\Tr_d(G_{vw}^m)=2dg(m,4)-2dg(m,5)=2d(m-1)^{(m-1)(M-5)-1}m^{4(m-2)}
((m-1)^{m-1}-m^{m-2}).$$
Clearly, there exists $m=m_0$ (e.g. $m_0=3$) such that $\Tr_d(G_{uw}^m)-\Tr_d(G_{vw}^m)\ne 0$, which implies that
$G_{uw}^{m_0}$ is not cospectral with $G_{vw}^{m_0}$.
The result follows.
\end{proof}


Finally we note that the smallest cospectral pair of graphs  are  $G_1:=K_{1,4}$ (also $S_5$ in Fig. \ref{tree5}) and $G_2:=C_4+K_1$.
Consider $\Tr_d(G_1^m)$ and $\Tr_d(G_2^m)$ with $d/m=2$.
By Eq. (\ref{Bu2}),
$$\Tr_d(G_1^m) =\sum_{k=1}^{2} d(m-1)^{(m-1)(4-k)}m^{k(m-2)} \sum_{\hat{T} \in \mathbf{T}_k} \tilde{c}_{2}(\hat{T}) N_{G_1}(\hat{T}).$$
Noting that $V(G_2^m)=V(C_4^m) \cup V(K_1)$, by Eq. (\ref{main3}) we have
$$\Tr_{d}(G_2^m) =\sum_{k=1}^{2} d(m-1)^{(m-1)(4-k)}m^{k(m-2)} \sum_{\hat{T} \in \mathbf{T}_k} \tilde{c}_{2}(\hat{T}) N_{G_2}(\hat{T}).$$
As $N_{G_1}(P_2)=N_{G_2}(P_2)=4$ and $N_{G_1}(P_3)=6 \ne N_{G_2}(P_3)=4$,
we have $\Tr_d(G_1^m) \ne \Tr_{d}(G_2^m)$.
So $K_{1,4}$ and $C_4+K_1$ are not high-order cospectral pair.

We consider the rooted products $G_1(P_n)$ and $G_2(P_n)$, where $P_n$ has one of its pendent vertices as the root and $n \ge 2$.
By Lemma \ref{root}, $G_1(P_n)$ and $G_2(P_n)$ are cospectral.
As $N_{G_1(P_n)}(P_3)-N_{G_2(P_n)}(P_3)=2 \ne 0$,
by a similar discussion we have
$\Tr_d((G_1(P_n))^m) \ne \Tr_{d}((G_2(P_n))^m)$ for $d/m=2$.
So $G_1(P_n)$ and $G_2(P_n)$ are also not high-order cospectral pair.

\begin{cor}\label{fin}
Let $G_1:=K_{1,4}$ and $G_2:=C_4+K_1$, and let $P_n$ be a path on $n$ vertices with one of its pendent vertices as root.
Then $G_1(P_n)$ and $G_2(P_n)$ are  not high-order cospectral pair.
\end{cor}

\section{Conclusion}
In this paper, we investigate the spectral characterization of unicyclic graphs by the high-ordered spectra of unicyclic graphs.
It is seen that the high-ordered spectra can recognize more structural information than the usual spectra, which implies that high-ordered spectra may have potential value on the graph isomorphism problem.
We also find there are two questions for further study on high-ordered spectral characterization of graphs.

(1) Recall that two graphs $G_1$ and $G_2$ are high-ordered cospectral if the spectrum of $G_1^m$ is the same as that of $G_2^m$ for all $m \ge 2$.
In fact, from the discussion in Section \ref{SCUG} (e.g. Lemmas \ref{NGen} and \ref{girth}), we only need  finitely many $m$'s such that $G_1$ and $G_2$ are $m$-ordered cospectral.
So, could we give an upper bound for the $m$ on the definition of high-ordered cospectral?
If it does, we will save times on comparing the high-ordered spectra of two graphs.

In addition, it is known that $G_1^m$ is cospectral with $G_2^m$ if and only if $\Tr_d(G_1^m)=\Tr_d(G_2^m)$ for all $d$ or $d=1,2,\ldots, n(m-1)^{n-1}$, where $n$ is the number of vertices of $G_1^m$ or $G_2^m$.
If $m \ge 3$, then $\Tr_d(G^m) \ne 0$ only if $m \mid d$.
As seen in Lemmas \ref{NGen} and \ref{CosT}, Corollaries \ref{CosG0}, \ref{CosG}, \ref{g34}, \ref{ominus} and \ref{fin}, we care more about $d/m$ than $d$.
This also can be found from the definition $\tilde{c}_{d/m}(\hat{T})$ in Eq. (\ref{cdom}).

(2) It is harder to compute the spectra or the characteristic polynomials of uniform hypergraphs than graphs,
as it is closely related the computation of resultants.
Though Cardoso et al. \cite{CardHT} gave a method to get the eigenvalues of the power $G^m$ from the eigenvalues of $G$ (see Lemma \ref{Card}), we still do not know the multiplicities of the eigenvalues of $G^m$.
So, how to compute the spectra of the power hypergraphs is a key question.
Fortunately, we can use traces instead of spectra to recognize high-ordered cospectral graphs, and can compute the traces of hypergraphs with simple structure by Corollary \ref{TrV}.


\begin{thebibliography}{99}

\bibitem{EB1987} T. van Aardenne-Ehrenfest, N. G. de Bruijn, Circuits and trees in oriented linear graphs, In: I. Gessel, G. C.Rota (Eds), \emph{Classic Papers in Combinatorics}, Birkh\"auser Boston, Boston, 1987, pp. 149-163.

\bibitem{BL}S. Bai, L. Lu, A bound on the spectral radius of hypergraphs with $e$ edges, \emph{Linear Algebra Appl.}, 549(2018), 203-218.



\bibitem{CardHT} K. Cardoso, C. Hoppen, V. Trevisan, The spectrum of a class of uniform hypergraphs, \emph{Linear Algebra Appl.}, 590 (2020), 243-257.

\bibitem{CPZ1} K. C. Chang, K. Pearson, T. Zhang, Perron-Frobenius theorem for nonnegative tensors, \emph{Commu. Math. Sci.}, 6 (2008), 507-520.


\bibitem{CBZ2022} L. Chen, C. Bu, J. Zhou, Spectral moments of hypertrees and their applications, \emph{Linear  Multilinear Algebra}, 2021, DOI: 10.1080/03081087.2021.1953431.


 \bibitem{CSBu} L. Chen, L. Sun, C. Bu, High-ordered spectral characterizations of graphs, available at arXiv: 2111.03877.


%
%
\bibitem{CC2021} G. Clark, J. Cooper, A Harary-Sachs theorem for hypergraphs, \emph{J. Combin. Theory Ser. B}, 149 (2021), 1-15.
%
\bibitem{CD2012} J. Cooper, A. Dutle, Spectra of uniform hypergraphs, \emph{Linear Algebra Appl.}, 436 (2012), 3268-3292.

\bibitem{ColS}L. Collatz, U. Sinogowitz
Spektren endlicher Grafen, \emph{Abh. Math. Semin. Univ. Hambg.}, 21 (1957),  63-77.

\bibitem{Cve} D. Cvetkovi\'c, P. Rowlinson, Spectra of unicyclic graphs,  Graphs  Combin., 3(1)1987, 7-23.

\bibitem{DamH1} E. R. van Dam, W. H. Haemers
Which graphs are determined by their spectrum?
\emph{Linear Algebra Appl.}, 373 (2003), 241-272.

\bibitem{DamH2} E. R. van Dam, W. H. Haemers
Developments on spectral characterizations of graphs
\emph{Discrete Math.}, 309 (2009), 576-586.




\bibitem{FBH} Y.-Z. Fan, Y.-H. Bao, T. Huang, Eigenvariety of nonnegative symmetric weakly irreducible tensors associated with spectral radius and its application to hypergraphs, \emph{Linear Algebra Appl.}, 564 (2019), 72-94.

\bibitem{FHBproc} Y.-Z. Fan, T. Huang, Y.-H. Bao,
The dimension of eigenvariety of nonnegative tensors associated with spectral radius, \emph{Proc. Amer. Math. Soc.}, 150 (2022), 2287-2299.


\bibitem{FHB} Y.-Z. Fan, T. Huang, Y.-H. Bao, C. L. Zhuan-Sun, Y.-P. Li,
The spectral symmetry of weakly irreducible nonnegative tensors and connected hypergraphs. \emph{Trans. Amer. Math. Soc.}, 372(3) (2019), 2213-2233.


\bibitem{FLW} Y.-Z. Fan, M. Li, Y. Wang, The cyclic index of adjacency tensor of generalized power hypergraphs, \emph{Discrete Math.}, 344 (2021), 112329.

\bibitem{FanTPL} Y.-Z. Fan, Y.-Y. Tan, X.-X. Peng, A.-H. Liu, Maximizing spectral radii of uniform hypergraphs with few edges, \emph{Discuss. Math. Graph Theory}, 36(4)(2016), 845-856.


\bibitem{FTL} Y.-Z. Fan, M.-Y. Tian, M. Li, Fan,
The stabilizing index and cyclic index of the coalescence and Cartesian product of uniform hypergraphs, \emph{J. Combin. Theory Ser. A}, 185 (2022), 105537.


\bibitem{Fan2022} Y.-Z. Fan, Y. Yang, C.-M. She, J. Zheng, Y.-M. Song, H.-X. Yang, The trace and Estrada index of uniform hypergraphs with cut vertices, available at arXiv: 2205.15502.


\bibitem{FGH} S. Friedland, S. Gaubert, L. Han, Perron-Frobenius theorem for nonnegative multilinear forms and extensions, \emph{Linear Algebra Appl.}, 438 (2013), 738-749.

\bibitem{GCH2022} G. Gao, A. Chang, Y. Hou, Spectral radius on linear $r$-graphs without expanded $K_{r+1}$, \emph{SIAM J. Discrete. Math.}, 36(2)(2022), 1000-1011.


\bibitem{GodM2} C. D. Godsil, B. D. McKay, A new graph product and its spectrum, \emph{Bull., Austral. Math. Soc.}, 18 (1978), 21-28.

\bibitem{GodM} C.D. Godsil, B.D. McKay
Constructing cospectral graphs, \emph{Aequationes Math.}, 25 (1982), 257-268.

\bibitem{GunP} Hs. H. G\"unthard, H. Primas,
Zusammenhang von Graphentheorie und MO-Theorie von Molekeln mit Systemen konjugierter Bindungen,
\emph{Helv. Chim. Acta}, 39 (1956), 1645-1653.

 \bibitem{HQS} S. Hu, L. Qi, J. Shao, Cored hypergraphs, power hypergraphs and their Laplacian eigenvalues,
\emph{Linear Algebra Appl.}, 439(2013), 2980-2998.


\bibitem{KLM2014} P. Keevash, J. Lenz, D. Mubayi, Spectral extremal problems for hypergraphs, \emph{SIAM J. Discrete Math.}, 28(4)(2014), 1838-1854.


\bibitem{Lim}L.-H. Lim, Singular values and eigenvalues of tensors: a variational approach, in \emph{Computational Advances in Multi-Sensor Adapative Processing}, 2005 1st IEEE International Workshop, IEEE, Piscataway, NJ, 2005, pp. 129-132.




\bibitem{MS2011} A. Morozov, Sh. Shakirov, Analogue of the identity Log Det = Trace Log for resultants, \emph{J. Geom. Phys.}, 61(3) (2011), 708-726.

\bibitem{Lep} M. Lepovi\'c, Some results on starlike trees and sunlike graphs, \emph{J. Appl. Math. Comput.}, 11(2003), 109-123.

\bibitem{LSQ} H. Li, J. Shao, L. Qi, The extremal spectral radii of $k$-uniform supertrees, \emph{J. Combin. Optim.}, 32(2016), 741-764.

\bibitem{LKS2018} L. Liu, L. Kang, E. Shan, On the irregularity of uniform hypergraphs, \emph{European J. Combin.}, 71(2018), 22-32.

\bibitem{LiuW}X. Liu, S. Wang, Y. Zhang, X. Yong, On the spectral characterization of some unicyclic graphs, \emph{Discrete Math.}, 311(2011), 2317-2336.

\bibitem{LM}L. Lu, S. Man, Connected hypergraphs with small spectral radius. \emph{Linear Algebra Appl.}, 509 (2016), 206-227.

\bibitem{Qi} L. Qi, Eigenvalues of a real supersymmetric tensor, \emph{J. Symbolic Comput.}, 40 (2005), 1302-1324.

\bibitem{Sch} A. J. Schwenk
Almost all trees are cospectral, In: F. Harary (Ed.), \emph{New Directions in the Theory of Graphs}, Academic Press, New York, 1973, pp. 275-307.

%
\bibitem{SQH2015}J.-Y. Shao, L. Qi, S. Hu, Some new trace formulas of tensors with applications in
spectral hypergraph theory, \emph{Linear Multilinear Algebra}, 63 (2015), 971-992.

\bibitem{TS1941} W. T. Tutte, C. A. B. Smith, On unicursal paths in a network of degree $4$, \emph{Amer. Math. Monthly}, 48(4) (1941), 233-237.

\bibitem{Wang} W. Wang, A simple arithmetic criterion for graphs being determined by their generalized spectra, \emph{J. Combin. Theory Ser. B}, 122 (2017), 438-451.

\bibitem{WangX1} W. Wang, C. X. Xu
A sufficient condition for a family of graphs being determined by their generalized spectra, \emph{European J. Combin.}, 27 (2006), 826-840.

\bibitem{WangX2} W. Wang, C. X. Xu
An excluding algorithm for testing whether a family of graphs are determined by their generalized spectra,
\emph{Linear Algebra Appl.}, 418 (2006), 62-74.


\bibitem{YY1}  Y. Yang, Q. Yang, Further results for Perron-Frobenius theorem for nonnegative tensors, \emph{SIAM J Matrix Anal. Appl.}, 31 (5) (2010), 2517-2530.

\bibitem{YY2} Y. Yang, Q. Yang, Further results for Perron-Frobenius theorem for nonnegative tensors II, \emph{SIAM J Matrix Anal. Appl.}, 32 (4) (2011), 1236-1250.

\bibitem{YY3} Y. Yang, Q. Yang, On some properties of nonnegative weakly irreducible tensors, arXiv: 1111.0713v2.

\bibitem{ZKSB} W. Zhang, L. Kang, E. Shan, Y. Bai, The spectra of uniform hypertrees. \emph{Linear Algebra Appl.}, 533(2017), 84-94.


\bibitem{Zhou} J. Zhou, L. Sun, W. Wang, C. Bu, Some spectral properties of uniform hypergraphs, \emph{Electron. J. Combin.}, 21 (2014), \#P4.24.


\end{thebibliography}
\end{document}